\documentclass[10pt, a4paper, reqno]{amsart}
\usepackage[foot]{amsaddr}

\usepackage{amsthm,mathrsfs,amsmath,amscd,enumerate,enumitem, amscd, xcolor, url}
\usepackage[colorlinks]{hyperref}
\usepackage[bbgreekl]{mathbbol}
\usepackage{amsfonts}
\usepackage{amssymb}
\usepackage[utf8]{inputenc}
\usepackage{mathtools}

\DeclareSymbolFontAlphabet{\amsbb}{bbold}
\DeclareSymbolFontAlphabet{\mathbb}{AMSb}%

\newtheorem{thm}{Theorem}[section]

\newtheorem{lem}[thm]{Lemma}

\theoremstyle{definition}

\theoremstyle{remark}
\newtheorem{rem}[thm]{Remark}

\DeclareMathOperator{\supp}{supp}

\newcommand{\BMO}{\textup{BMO}}

\numberwithin{equation}{section}
\allowdisplaybreaks

\setlist[enumerate,1]{label=(\alph*)}

\begin{document}
    \title[Endpoint estimates for operators in the Laguerre setting]
    {Endpoint estimates for harmonic analysis operators associated with Laguerre polynomial expansions}

    \author[J. J. Betancor]{Jorge J. Betancor$^1$}
    \address{$^1$Departamento de An\'alisis Matem\'atico, Universidad de La Laguna,\newline
        Campus de Anchieta, Avda. Astrof\'isico S\'anchez, s/n,\newline
        38721 La Laguna (Sta. Cruz de Tenerife), Spain}
    \email{$^1$jbetanco@ull.es}

    \author[E. Dalmasso]{Estefan\'ia Dalmasso$^{2,*}$}
    \address{$^2$Instituto de Matem\'atica Aplicada del Litoral, UNL, CONICET, FIQ.\newline Colectora Ruta Nac. Nº 168, Paraje El Pozo,\newline S3007ABA, Santa Fe, Argentina}
    \email[Corresponding Author]{$^*$edalmasso@santafe-conicet.gov.ar \textnormal{(Corresponding  author)}}

    \author[P. Quijano]{Pablo Quijano$^{2,\dagger}$}
    \email{$^{\dagger}$pquijano@santafe-conicet.gov.ar}

    \author[R. Scotto]{Roberto Scotto$^3$}
    \address{$^3$Universidad Nacional del Litoral, FIQ.\newline Santiago del Estero 2829,\newline S3000AOM, Santa Fe, Argentina}
    \email{$^3$roberto.scotto@gmail.com}

    \date{\today}
    \subjclass{42B15, 42B20, 42B25, 42B35}

    \keywords{Hardy spaces, BMO spaces, endpoint estimates, Laguerre polynomials}

    \begin{abstract}
        In this paper we give a criterion to prove boundedness results for several operators from $H^1((0,\infty),\gamma_\alpha)$ to $L^1((0,\infty),\gamma_\alpha)$ and also from $L^\infty((0,\infty),\gamma_\alpha)$ to $\BMO((0,\infty),\gamma_\alpha)$, with respect to the probability measure $d\gamma_\alpha (x)=\frac{2}{\Gamma(\alpha+1)} x^{2\alpha+1} e^{-x^2} dx$ on $(0,\infty)$ when ${\alpha>-\frac12}$. We shall apply it to establish endpoint estimates for Riesz transforms, maximal operators, Littlewood-Paley functions, multipliers of Laplace transform type, fractional integrals and variation operators in the Laguerre setting.
    \end{abstract}
    \maketitle

\section{Introduction and main results}

We will consider harmonic analysis operators associated with Laguerre polynomial expansions and establish some endpoint estimates for them. More precisely, we shall prove their boundedness from $H^1((0,\infty),\gamma_\alpha)$ to $L^1((0,\infty),\gamma_\alpha)$ and from $L^\infty((0,\infty),\gamma_\alpha)$ to $\BMO((0,\infty),\gamma_\alpha)$, where $d\gamma_\alpha (x)=\frac{2}{\Gamma(\alpha+1)} x^{2\alpha+1} e^{-x^2} dx$ on $(0,\infty)$, for $\alpha>-\frac12$. 

We first recall the definitions and the main properties of Hardy and BMO spaces in our setting. We introduce the function $m(x)=\min\left\{1,\frac{1}{x}\right\}$,~$x\in (0,\infty)$. Given $a>0$, an interval $(x-r,x+r)$, with $0<r\le x$, is said to be in~$\mathcal{B}_a$ when $0<r\le am(x)$. We denote $I(x,r):=(x-r,x+r)\cap (0,\infty)$, for every $x,r\in (0,\infty)$. Clearly, the measure $\gamma_\alpha$ is not doubling but it has the doubling property on the intervals of each family $\mathcal{B}_a$, that is, there exists $C_a>0$ for which
\[
\gamma_\alpha(I(x,2r))\le C_a\gamma_\alpha(I(x,r)),
\]
provided that $I(x,r)\in \mathcal{B}_a$.

Let $1<q\le \infty$. We say that a measurable function $b$ defined on $(0,\infty)$ is an $(a,q,\alpha)$-atom when $b(x)=1$ for every $x\in (0,\infty)$, or there exists $0<r_0\le x_0$ such that $I(x_0,r_0)\in \mathcal{B}_a$ and the following properties are satisfied:
\begin{enumerate}[label=(\roman*)]
    \item $\supp(b)\subset I(x_0,r_0)$;
    \item $\|b\|_{L^q((0,\infty),\gamma_\alpha)}\le \gamma_\alpha(I(x_0,r_0))^{\frac{1}{q}-1}$, being $\frac{1}
{q}=0$ when $q=\infty$;
    \item $\int_0^\infty b(y)d\gamma_\alpha(y)=0$.
\end{enumerate}
A function $f\in L^1((0,\infty),\gamma_\alpha)$ is said to be in $H^{1,q}_a((0,\infty),\gamma_\alpha)$ when $f=\sum_{j=0}^\infty \lambda_jb_j$ where, for every $j\in \mathbb{N}$, $b_j$ is an $(a,q,\alpha)$-atom and $\lambda_j\in \mathbb{C}$ being $\sum_{j=0}^\infty|\lambda_j|<\infty.$ If $f\in H^{1,q}_a((0,\infty),\gamma_\alpha)$ we define
\[
\|f\|_{H^{1,q}_a((0,\infty),\gamma_\alpha)}= \inf\,\sum_{j=0}^\infty|\lambda_j|.
\]
Here the infimum is taken over all the sequences $\{\lambda_j\}_{j\in \mathbb{N}}$ of complex numbers such that $\sum_{j=0}^\infty|\lambda_j|<\infty$ and $f=\sum_{j=0}^\infty \lambda_jb_j$, where $b_j$ is an $(a,q,\alpha)$-atom, for every $j\in\mathbb{N}$. The space $H^{1,q}_a((0,\infty),\gamma_\alpha)$ is endowed with the topology associated with the norm $\|\cdot\|_{H^{1,q}_a((0,\infty),\gamma_\alpha)}$. Thus, $H^{1,q}_a((0,\infty),\gamma_\alpha)$ is a Banach space.

These Hardy spaces $H^{1,q}_a((0,\infty),\gamma_\alpha)$ were introduced and studied in \cite{BDQS-2}, where the authors gave a characterization of $H^{1,q}_a((0,\infty),\gamma_\alpha)$ by using a local maximal function. The space $H^{1,q}_a((0,\infty),\gamma_\alpha)$ actually does not depend on $a$ and $q$ (\cite[Theorem~1.1]{BDQS-2}). To simplify, we write  $H^1((0,\infty),\gamma_\alpha)$ instead of $H^{1,q}_a((0,\infty),\gamma_\alpha)$, $a>0$ and $1<q\le \infty$.

We shall now define our $\BMO$ space. Let $a>0$. We say that a function ${f\in L^1((0,\infty),\gamma_\alpha)}$ is in $\BMO_a((0,\infty),\gamma_\alpha)$ when
\[
\|f\|_{*,a,\alpha}:= \sup_{I\in \mathcal{B}_a}\frac{1}{\gamma_\alpha(I)}\int_I|f(y)-f_I|d\gamma_\alpha(y)<\infty,
\]
where $f_I:=\frac{1}{\gamma_\alpha(I)}\int_If(y)d\gamma_\alpha(y)$, for each $I\in \mathcal{B}_a$. We also define, for every ${f\in \BMO_a((0,\infty),\gamma_\alpha)}$,
\[
\|f\|_{\BMO_a((0,\infty),\gamma_\alpha)}=\|f\|_{L^1((0,\infty),\gamma_\alpha)}+\|f\|_{*,a,\alpha}.
\]
The space $\left(\BMO_a((0,\infty),\gamma_\alpha),\|\cdot\|_{\BMO_a((0,\infty),\gamma_\alpha)}\right)$ is a Banach space when the functions differing by a constant are identified.

As it happens with the Hardy space, the $\BMO_a((0,\infty),\gamma_\alpha)$ space does not depend on the parameter $a$. Hence, we may write $\BMO((0,\infty),\gamma_\alpha)$ instead of $\BMO_a((0,\infty),\gamma_\alpha)$, and $\|\cdot\|_{*,\alpha}$ instead of $\|\cdot\|_{*,a,\alpha}$.

We now state the main properties of the space $\BMO((0,\infty),\gamma_\alpha)$ that can be proved as the corresponding properties for the $\BMO$ space associated with the Gaussian measure (see \cite{MM}).

\begin{thm} \label{ThBMO} 
Let $\alpha>-\frac12$.
\begin{enumerate}
    \item \label{dual} The dual space of $H^1((0,\infty),\gamma_\alpha)$ can be identified with $\BMO((0,\infty),\gamma_\alpha)$.
    \item \label{JNineq} (John-Nirenberg type inequality) There exist $c,C>0$ such that, for every $f\in \BMO((0,\infty),\gamma_\alpha)$ and every $I\in \mathcal{B}_1$,
\[\gamma_\alpha\left(\{x\in I:\,|f(x)-f_I|>\lambda\}\right)\le C \exp
 \left(-\frac{c\lambda}{\|f\|_{*,\alpha}}\right)\gamma_\alpha(I).\]
\end{enumerate}
\end{thm}

The key result for the aforementioned endpoint estimates for the harmonic analysis operators associated with Laguerre expansions will be a boundedness criterion, in the spirit of a similar result given in \cite{MM} on the Gaussian setting.

We shall be dealing with the following two classes of operators $T$, namely:
\begin{enumerate}[label=(\Roman*)]
    \item $T$ is a linear operator defined on  $L^2((0,\infty),\gamma_\alpha)$ into the space of measurable functions on $(0,\infty)$ such that, for a certain measurable function $K:((0,\infty)\times (0,\infty))\setminus D\rightarrow \mathbb{C}$, where $D=\{(x,x):x\in (0,\infty)\}$,
\[
T(f)(x)=\int_0^\infty K(x,y)f(y)d\gamma_\alpha(y),\quad\,x\in (0,\infty)\setminus \supp(f),
\]
with $f\in L^2((0,\infty),\gamma_\alpha)$.
\item For every $t>0$, $T_t$ is a linear operator defined on  $L^2((0,\infty),\gamma_\alpha)$ into the space of measurable functions such that, for a certain measurable function $K_t:(0,\infty)\times (0,\infty)\rightarrow \mathbb{C}$,
\[
T_t(f)(x)=\int_0^\infty K_t(x,y)f(y)d\gamma_\alpha(y),\quad\,x\in (0,\infty),
\]
where $f\in L^2((0,\infty),\gamma_\alpha)$. Suppose now that $(X,\|\cdot\|_X)$ is a Banach space of measurable complex functions defined on $(0,\infty)$, that the function
\begin{align*}
((0,\infty)\times (0,\infty))\setminus D &\rightarrow X,\\
(x,y)&\mapsto K_{(\cdot)}(x,y),    
\end{align*}
is $X$-strongly measurable and that, for every $f\in L^\infty((0,\infty),\gamma_\alpha)$, the integral $\int_0^\infty \|K_{(\cdot)}(x,y)\|_Xf(y)d\gamma_\alpha(y)<\infty$, for every $x\in (0,\infty)\setminus \supp(f)$. The operator $T$ is defined by
\[
T(f)=\|T_t(f)\|_X, \quad\,f\in L^2((0,\infty),\gamma_\alpha).
\]
\end{enumerate}
Note that (II)-type operators are special cases of the so-called almost linear operators.

The boundedness criterion can be now stated.

\begin{thm}\label{criterion} Let $\alpha>-\frac12$. Suppose $T$ is an operator as above, that is, $T$ is of (I)-type or (II)-type.
\begin{enumerate}
    \item \label{Linf-BMO} Assume that $T$ is bounded on $L^2((0,\infty),\gamma_\alpha)$ and
\begin{equation}\label{C1}
\sup_{I\in \mathcal{B}_1}\sup_{x,z\in I}\int_{(2I)^c}\|K_t(x,y)-K_t(z,y)\|_X d\gamma_\alpha(y)<\infty,
\end{equation}
where $X=\mathbb{C}$ when the operator $T$ is of (I)-type. Then, $T$ is bounded from $L^\infty((0,\infty),\gamma_\alpha)$ into $\BMO((0,\infty),\gamma_\alpha)$.
\item \label{H1-L1} Assume that $T$ is bounded on $L^2((0,\infty),\gamma_\alpha)$ and it can be extended from $L^2((0,\infty),\gamma_\alpha)$ to $L^1((0,\infty),\gamma_\alpha)$ as a bounded operator from $L^1((0,\infty),\gamma_\alpha)$ into $L^{1,\infty}((0,\infty),\gamma_\alpha)$. Also, suppose that
\begin{equation}\label{C2}
\sup_{I\in \mathcal{B}_1}\sup_{y,z\in I}\int_{(2I)^c}\|K_t(x,y)-K_t(x,z)\|_Xd\gamma_\alpha(x)<\infty,
\end{equation}
where, as above, $X=\mathbb{C}$ when $T$ is of (I)-type. Then, $T$ can be extended from $H^1((0,\infty),\gamma_\alpha)\cap L^2((0,\infty),\gamma_\alpha)$ to $H^1((0,\infty),\gamma_\alpha)$ as a
bounded operator from $H^1((0,\infty),\gamma_\alpha)$ into $L^1((0,\infty),\gamma_\alpha)$.
\end{enumerate}
\end{thm}

\begin{rem}If the kernel $K_t$ has certain regularity, conditions \eqref{C1} and \eqref{C2} can be deduced from
\begin{equation}\label{C1'}
    \sup_{I\in \mathcal{B}_1} r_I\sup_{x\in I}\int_{(2I)^c}\|\partial_x K_t(x,y)\|_X d\gamma_\alpha(y)<\infty,
\end{equation}
and
\begin{equation}\label{C2'}
    \sup_{I\in \mathcal{B}_1}r_I\sup_{y\in I}\int_{(2I)^c}\|\partial_y K_t(x,y)\|_X d\gamma_\alpha(x)<\infty,
\end{equation}
respectively, being $r_I$ the radius of the interval $I$.
\end{rem}

We now present the setting and the operators we shall be dealing with. Let $\alpha>-\frac12$. For every $k\in \mathbb{N}$, the Laguerre polynomial $L_k^\alpha$ of order $\alpha$ and degree $k$ is defined as in (\cite{Leb})
\[
L_k^\alpha(x)=\sqrt{\frac{\Gamma(\alpha+1)}{\Gamma(\alpha+k+1)k!}}e^xx^{-\alpha}\frac{d^k}{dx^k}\left(e^{-x}x^{\alpha+1}\right),\quad x\in (0,\infty).
\]
We consider the Laguerre operator $\widetilde{\Delta}_\alpha$ given by
\[
\widetilde{\Delta}_\alpha(f)(x)=-\frac{1}{4}\left(\frac{d^2}{dx^2}+\left(\frac{2\alpha+1}{x}-2x\right)\frac{d}{dx}\right)f(x),\quad\,f\in C^2(0,\infty).
\]
By defining, for every $k\in \mathbb{N}$, $\mathcal{L}_k^\alpha(x)=L_k^\alpha(x^2)$, $x\in (0,\infty)$, the sequence
$\{\mathcal{L}_k^\alpha\}_{k\in \mathbb{N}}$ is an orthonormal basis in $L^2((0,\infty),\gamma_\alpha)$. Furthermore, $\widetilde{\Delta}_\alpha\mathcal{L}_k^\alpha=k\mathcal{L}_k^\alpha$, for any
$k\in \mathbb{N}$.

For every $f\in L^1((0,\infty),\gamma_\alpha)$, we set
\[
c_k^\alpha(f)=\int_0^\infty f(x)\mathcal{L}_k^\alpha(x)d\gamma_\alpha(x),\quad\,k\in \mathbb{N}.
\]
We consider the operator $\Delta_\alpha$ given by
\[
\Delta_\alpha(f)=\sum_{k=0}^\infty kc_k^\alpha(f)\mathcal{L}_k^\alpha, \quad\,f\in D(\Delta_\alpha),
\]
where $D(\Delta_\alpha)=\{f\in L^2((0,\infty),\gamma_\alpha):\,\sum_{k=0}^\infty|kc_k^\alpha(f)|^2<\infty\}$. We have that $\Delta_\alpha(f)=\widetilde{\Delta}_\alpha(f)$ for every $f\in C_c^\infty(0,\infty)$, the space of smooth functions with compact support in $(0,\infty)$. The operator $\Delta_\alpha$ is a self-adjoint and positive operator. Moreover, the operator $-\Delta_\alpha$ generates a semigroup of operators $\{W_t^\alpha\}_{t>0}$ in $L^2((0,\infty),\gamma_\alpha)$ where, for every $t>0$,
\[
W_t^\alpha(f)=\sum_{k=0}^\infty e^{-kt}c_k^\alpha(f)\mathcal{L}_k^\alpha,\quad f\in L^2((0,\infty),\gamma_\alpha).
\]
By using the Hille-Hardy formula (\cite[(4.17.6)]{Leb})
we can write, for every $x,y,t>0$
\begin{equation}\label{Wtalfa-ker}
\sum_{k=0}^\infty e^{-kt}\mathcal{L}_k^\alpha(x)\mathcal{L}_k^\alpha(y)=\frac{\Gamma(\alpha+1)}{1-e^{-t}}\left(e^{-t/2}xy\right)^{-\alpha}I_\alpha\left(\frac{2e^{-t/2}xy}{1-e^{-t}}\right)e^{-\frac{e^{-t}(x^2+y^2)}{1-e^{-t}}},    
\end{equation}
where $I_\alpha$ denotes the modified Bessel function of the first kind and order $\alpha$.

We get, for every $f\in L^2((0,\infty),\gamma_\alpha)$ and $t>0$,
\begin{equation}\label{sem}
W_t^\alpha(f)(x)=\int_0^\infty W_t^\alpha(x,y)f(y)d\gamma_\alpha(y),\quad\,x\in (0,\infty),
\end{equation}
being $W_t^\alpha(x,y)$ the right-hand side of \eqref{Wtalfa-ker}, for $x,y,t\in (0,\infty)$.

The integral in (\ref{sem}) is absolutely convergent for every $f\in L^p((0,\infty),\gamma_\alpha)$, with $1\le p<\infty$. Furthermore, by defining, for every $t>0$, $W_t^\alpha$ on $L^p((0,\infty),\gamma_\alpha)$,
$1\le p<\infty$, by (\ref{sem}) $\{W_t^\alpha\}_{t>0}$ turns out to be a symmetric diffusion semigroup in the sense of Stein (\cite{StLP}).

We define the subordinated Poisson semigroup $\{P_t^\alpha\}_{t>0}$ associated with the Laguerre operator $\Delta_\alpha$ by
\begin{equation}
    \label{subor}
P_t^\alpha(f)=\frac{t}{2\sqrt{\pi}}\int_0^\infty e^{-t^2/(4u)}W_u^\alpha(f)\frac{du}{u^{3/2}},\quad\,t>0.
\end{equation}
Thus, $\{P_t^\alpha\}_{t>0}$ is a symmetric diffusion semigroup in the sense of Stein (\cite{StLP}) as well.

Given $k\in \mathbb{N}$, the maximal operator $P_{*,k}^\alpha$ is given by
\[
P_{*,k}^\alpha(f)(x)=\sup_{t>0}\left|t^{k}\partial^k_tP_t^\alpha(f)(x)\right|,\quad x\in (0,\infty).
\]
According to \cite[p.~73]{StLP} we have that $P_{*,0}^\alpha$ is bounded on $L^p((0,\infty),\gamma_\alpha)$, for every $1<p<\infty$. The maximal operator $W_*^\alpha$ associated with the heat semigroup $\{W_t^\alpha\}_{t>0}$ was studied by Muckenhoupt (\cite{Mu2}) (see also \cite{Di}). From Muckenhoupt's results it follows that $P_{*,0}^\alpha$ is bounded from $L^1((0,\infty),\gamma_\alpha)$ into $L^{1,\infty}((0,\infty),\gamma_\alpha)$. The Laguerre semigroup $\{W_t^\alpha\}_{t>0}$ can be extended to complex values of the parameter $t$. The corresponding maximal operator was studied in \cite{Sa1}. From \cite[Remark 4.3]{LeMX2} the operator $P_{*,k}^\alpha$ is bounded on $L^p((0,\infty),\gamma_\alpha)$, for every $1<p<\infty$ and $k\in \mathbb{N}$. By using \cite[Theorem~1.1]{BdL} we deduce that $P_{*,k}^\alpha$ is also bounded from $L^1((0,\infty),\gamma_\alpha)$ into $L^{1,\infty}((0,\infty),\gamma_\alpha)$, for every $k\in \mathbb{N}$.

The Littlewood-Paley functions associated with the semigroup $\{P_t^\alpha\}_{t>0}$ can be defined as follows. For every $n,k\in \mathbb{N}$ with $n+k\ge 1$, we consider the square function $g_{n,k}^\alpha$ given by
\[
g_{n,k}^\alpha(f)(x)=\left(\int_0^\infty \left|t^{n+k}\partial_x^n\partial_t^kP_t^\alpha(f)(x)\right|^2\frac{dt}{t}\right)^{1/2},\quad\,x\in (0,\infty).
\]
According to \cite[Chapter~4,~Section~6,~Corollary~1]{StLP}, for every $1<p<\infty$ and $k\in \mathbb{N}$, $k\ge 1$, $g_{0,k}^\alpha$ is bounded on $L^p((0,\infty),\gamma_\alpha)$. Moreover, in \cite[Theorem~1.2]{BdL} it was proved that $g_{0,k}^\alpha$ is bounded from $L^1((0,\infty),\gamma_\alpha)$ into $L^{1,\infty}((0,\infty),\gamma_\alpha)$, for every  $k\in \mathbb{N}$, $k\ge 1$. The authors proved in \cite[Theorem 1.1, (c)]{BDQS} that $g_{n,k}^\alpha$ is bounded on $L^p((0,\infty),\gamma_\alpha)$, for every $1<p<\infty$ and $n,k\in \mathbb{N}$ with $n+k\ge 1$. It is worth mentioning that Nowak (\cite[Theorems~6~and~7]{No}) established $L^p$-boundedness properties for Littlewood-Paley functions involving spatial derivatives ($n>0$) in other Laguerre contexts.

We now introduce the multipliers of Laplace transform type. A measurable function $M$ is said to be of Laplace transform type when
\[
M(x)=x\int_0^\infty e^{-xy}\phi(y)dy,\quad\,x\in (0,\infty),
\]
where $\phi\in L^\infty(0,\infty)$. Note that $M\in L^\infty(0,\infty)$ provided that $M$ is of Laplace transform type. We define the multiplier $T_M^\alpha$ by
\[
T_M^\alpha(f)=\sum_{k=1}^\infty M\left(\sqrt{k}\right)c_k^\alpha(f)\mathcal{L}_k^\alpha,\quad\,f\in L^2((0,\infty),\gamma_\alpha).
\]
Since $M\in L^\infty(0,\infty)$, $T_M^\alpha$ is bounded on $L^2((0,\infty),\gamma_\alpha)$. By \cite[Corollary~3,~p.~121]{StLP} $T_M^\alpha$ can be extended from $L^2((0,\infty),\gamma_\alpha)\cap L^p((0,\infty),\gamma_\alpha)$ to $L^p((0,\infty),\gamma_\alpha)$ as a bounded operator on $L^p((0,\infty),\gamma_\alpha)$, for every $1<p<\infty$. When $\beta>0$ and $\phi(y)=\frac{1}{\Gamma(1-2\beta i)}y^{-2\beta i}$, for $y\in (0,\infty)$, then $T_M^\alpha=\Delta_\alpha^{\beta i}$.

Let $\omega>0$. We define the negative power $\Delta_\alpha^{-\omega}$ of $\Delta_\alpha$ by
\[
\Delta_\alpha^{-\omega}(f)(x)=\sum_{k=1}^\infty k^{-\omega}c_k^\alpha(f)\mathcal{L}_k^\alpha,\quad\,f\in L^2((0,\infty),\gamma_\alpha).
\]
Thus, $\Delta_\alpha^{-\omega}$ is bounded on $L^2((0,\infty),\gamma_\alpha)$. Note that the function $M(x)=x^{-2\omega}$, for $x\in (0,\infty)$, is not of Laplace transform type, so $\Delta_\alpha^{-\omega}$ does not fall in the scope of multipliers of Laplace transform type. The operator $\Delta^{-\omega}_\alpha$ can be extended from $L^2((0,\infty),\gamma_\alpha)\cap L^p((0,\infty),\gamma_\alpha)$ to $L^p((0,\infty),\gamma_\alpha)$ as a bounded operator on $L^p((0,\infty),\gamma_\alpha)$, for every $1<p<\infty$ (see \cite[Lemma~2.2]{MN} for the Ornstein-Uhlenbeck operator case). However, $\Delta_\alpha^{-\omega}$ is not bounded from $L^1((0,\infty),\gamma_\alpha)$ into $L^{1,\infty}((0,\infty),\gamma_\alpha)$ (see \cite[Proposition~6.2]{GCMST} for the Ornstein-Uhlenbeck operator case).

Let $n\in \mathbb{N}$, $n\ge 1$. We define the $n$-th order Riesz transform $R_\alpha^n$ associated with $\Delta_\alpha$ as follows. For every $f\in L^2((0,\infty),\gamma_\alpha)$,
\[
R_\alpha^n(f)=\sum_{k=1}^\infty \frac{1}{k^{n/2}}c_k^\alpha(f)\frac{d^n}{dx^n}\mathcal{L}_k^\alpha.
\]
By \cite{NoS}, $R_\alpha^n$ is bounded on $L^2((0,\infty),\gamma_\alpha)$. For every  $f\in L^2((0,\infty),\gamma_\alpha)$ we have that
\[
R_\alpha^n(f)(x)=\int_0^\infty K_\alpha^n(x,y)f(y)d\gamma_\alpha(y),\quad \text{a.e. } x\in (0,\infty)\setminus \supp(f),
\]
where
\[
K_\alpha^n(x,y)=\frac{1}{\Gamma(n/2)}\int_0^\infty t^{-1+n/2}W_t^\alpha(x,y)dt,\quad x,y\in (0,\infty), x\neq y.
\]
By \cite[Theorem~1.1]{BFRS}, \cite[Theorem~1.1]{Sa4}, and \cite[Theorem~13]{No} we deduce that each $R_\alpha^n$ can be extended
from $L^2((0,\infty),\gamma_\alpha)\cap L^p((0,\infty),\gamma_\alpha)$ to $L^p((0,\infty),\gamma_\alpha)$ as a
bounded operator on $L^p((0,\infty),\gamma_\alpha)$, for every $1<p<\infty$.
Furthermore, $R_\alpha^1$ can be extended from $L^2((0,\infty),\gamma_\alpha)\cap L^1((0,\infty),\gamma_\alpha)$ to $L^1((0,\infty),\gamma_\alpha)$
as a bounded operator from $L^1((0,\infty),\gamma_\alpha)$ into $L^{1,\infty}((0,\infty),\gamma_\alpha)$ (see \cite{FSS2}).

Finally, we present variation operators. Let $\rho>2$. If $F$ is a complex valued and measurable function defined on $(0,\infty)$, we define the $\rho$-variation $V_\rho(\{F(t)\}_{t>0})$ as follows
\[
V_\rho(\{F(t)\}_{t>0})= \sup_{0<t_N<t_{N-1}<...<t_1,\; N\in \mathbb{N},\; N\ge 1}\left(\sum_{j=1}^{N-1}|F(t_j)-F(t_{j+1})|^\rho\right)^{1/\rho}.
\]
According to \cite[Corollary~6.1]{LeMX1} (see also \cite[Theorem~3.3]{JR}), the $\rho$-variation operator $V_\rho(\{t^k\partial_t^k P_t^\alpha\}_{t>0})$, $k\in \mathbb{N}$, is bounded on $L^p((0,\infty),\gamma_\alpha)$, ${1<p<\infty}$. Later, in \cite[Theorem~1.1]{BdL} it was proved that $V_\rho(\{t^k\partial_t^kP_t^\alpha\}_{t>0})$ is bounded from $L^1((0,\infty),\gamma_\alpha)$ into $L^{1,\infty}((0,\infty),\gamma_\alpha)$.

Other results for the above harmonic analysis operators in the $\Delta_\alpha$-setting can be found in \cite{AE,BDQS,FSS1,GIT,GLLNU,Sa2,Sa3}.

In the following we establish the endpoints inequalities for the operators we have just introduced.

\begin{thm}\label{endpoints} Let $\alpha>-\frac12$ and $\rho>2$.
\begin{enumerate}
    \item \label{Linf-BMO-op} The operators $P_{*,k}^\alpha$ for $k\in \mathbb N$, $g_{n,k}^\alpha$ for $n,k\in \mathbb{N}$ with $n+k\ge 1$, $R_\alpha^n$ for $n\in \mathbb{N}$, $n\ge 1$, $V_\rho(\{t^k\partial_t^k P_t^\alpha\}_{t>0})$ for $k\in \mathbb{N}$, $\Delta_\alpha^{-\omega}$ for $\omega>0$, and $T_M^\alpha$ where $M$ is of Laplace transform type, are all bounded from $L^\infty((0,\infty),\gamma_\alpha)$ into $\BMO((0,\infty),\gamma_\alpha)$.
    
    \item \label{H1-L1-op} The operators $P_{*,k}^\alpha$ for $k\in \mathbb N$, $g_{0,k}^\alpha$ for $k\in \mathbb{N}$, $k\ge 1$, $R_\alpha^1$, $V_\rho(\{t^k\partial_t^kP_t^\alpha\}_{t>0})$ for $k\in \mathbb{N}$, $\Delta_\alpha^{-\omega}$ for $\omega>0$, and $T_M^\alpha$ where $M$ is of Laplace transform type, can be extended from $L^2((0,\infty),\gamma_\alpha)\cap H^1((0,\infty),\gamma_\alpha)$ to $H^1((0,\infty),\gamma_\alpha)$ as bounded operators from $H^1((0,\infty),\gamma_\alpha)$ to $L^1((0,\infty),\gamma_\alpha)$.
\end{enumerate}
\end{thm}

We remark that endpoint inequalities for Riesz transforms and some spectral multipliers in the Ornstein-Uhlenbeck setting were proved in \cite{MM} and \cite{MMS-end}.

This paper is organized as follows. In Section~\ref{dem-criterion}, we prove Theorem~\ref{criterion}. In Section~\ref{lemmas-aux} we state several lemmas that will be useful when proving our main theorem in Section~\ref{mainthm}.

\section{Proof of Theorem~\ref{criterion}} \label{dem-criterion}

In order to prove~\ref{Linf-BMO} we can proceed as in the proof of \cite[Theorem~6.1,~(i)]{MM}. We sketch it here. Suppose that $T$ is a (II)-type operator associated with the Banach space $X$. When $T$ is an (I)-type operator the proof is similar. 

Let $f\in L^\infty((0,\infty),\gamma_\alpha)$ and $I\in \mathcal{B}_1$. We define $f_1=f\chi_{2I}$ and $f_2=f-f_1$. We have that $T_t(f)=T_t(f_1)+T_t(f_2)$, $t>0$. Thus, we can write
\begin{align*}
\int_I&\left|T(f)(x)-\|(T_t(f))_I\|_X \right| d\gamma_\alpha(x)\\
&\le \int_I\|T_t(f)(x)-(T_t(f))_I\|_X d\gamma_\alpha(x)\\
&\le\int_I\|T_t(f_1)(x)-(T_t(f_1))_I\|_Xd\gamma_\alpha(x)+\int_I\|T_t(f_2)(x)-(T_t(f_2))_I\|_Xd\gamma_\alpha(x)\\
&\le 2\int_I\|T_t(f_1)(x)\|_Xd\gamma_\alpha(x)+\int_I\|T_t(f_2)(x)-(T_t(f_2))_I\|_Xd\gamma_\alpha(x).
\end{align*}

Since $T$ is bounded on $L^2((0,\infty),\gamma_\alpha)$ and $\gamma_\alpha$ is doubling on $\mathcal B_1$, we get
\begin{align*}
\frac{1}{\gamma_\alpha(I)}\int_I\|T_t(f_1)(x)\|_Xd\gamma_\alpha(x)&\le \left(\frac{1}{\gamma_\alpha(I)}\int_I|T(f_1)(x)|^2d\gamma_\alpha(x)\right)^{1/2}\\
&\le C\left(\frac{1}{\gamma_\alpha(I)}\int_{2I} |f(x)|^2d\gamma_\alpha(x)\right)^{1/2}\\
&\le C\|f\|_{L^\infty((0,\infty),\gamma_\alpha)}.
\end{align*}
According to (\ref{C1}) and the properties of the Bochner integration we obtain for the other term
\begin{align*}
\frac{1}{\gamma_\alpha(I)}&\int_I\|T_t(f_2)(x)-(T_t(f_2))_I\|_Xd\gamma_\alpha(x)
\\&=\frac{1}{(\gamma_\alpha(I))^2}\int_I\left\|\int_I(T_t(f_2)(x)-T_t(f_2)(z))d\gamma_\alpha(z)\right\|_Xd\gamma_\alpha(x)\\
&\le \frac{1}{(\gamma_\alpha(I))^2}\int_I\int_I\left\|\int_{(2I)^c}(K_t(x,y)-K_t(z,y))f(y)d\gamma_\alpha(y)\right\|_Xd\gamma_\alpha(z)d\gamma_\alpha(x)\\
&\le \frac{1}{(\gamma_\alpha(I))^2}\int_I\int_I\int_{(2I)^c}\|K_t(x,y)-K_t(z,y)\|_X|f(y)|d\gamma_\alpha(y)d\gamma_\alpha(z)d\gamma_\alpha(x)\\
&\le \|f\|_{L^\infty((0,\infty),\gamma_\alpha)}\sup_{x,z\in I}\int_{(2I)^c}\|K_t(x,y)-K_t(z,y)\|_Xd\gamma_\alpha(y)\\
&\le C\|f\|_{L^\infty((0,\infty),\gamma_\alpha)}.
\end{align*}
We conclude that $T(f)\in \BMO((0,\infty),\gamma_\alpha)$ and \[\|T(f)\|_{\BMO((0,\infty),\gamma_\alpha)}\le C\|f\|_{L^\infty((0,\infty),\gamma_\alpha)}.\]

Now, we are going to prove~\ref{H1-L1}. In the proof of \cite[Theorem~6.1,~(ii)]{MM} Mauceri and Meda used the duality between Hardy and \BMO\ spaces. Here we cannot use duality arguments when we consider (II)-type operators. Assume that $T$ is a (II)-type operator associated with a Banach space $X$. For (I)-type operators the same proof works.

Suppose that $b$ is a $(1,2,\alpha)$--atom. If $b(x)=1$ for every $x\in (0,\infty)$, we have that
\[
\|T(b)\|_{L^1((0,\infty),\gamma_\alpha)}\le \|T(b)\|_{L^2((0,\infty),\gamma_\alpha)}\le C\|b\|_{L^2((0,\infty),\gamma_\alpha)}\le C.
\]
Assume now that $0<r_I\le x_I$ are such that $I=(x_I-r_I,x_I+r_I)\in \mathcal{B}_1$ and the following properties hold
\begin{enumerate}[label=(\roman*)]
    \item \label{sop-b} $\supp(b)\subset I$;
    \item \label{norm-b} $\|b\|_{L^2((0,\infty),\gamma_\alpha)}\le \gamma_\alpha(I)^{-\frac12}$;
    \item \label{int-b} $\int_0^\infty b(x)d\gamma_\alpha(x)=0$.
\end{enumerate}

We can write
\begin{align*}
\int_0^\infty|T(b)(x)|d\gamma_\alpha(x)&=\int_{(2I)\cap (0,\infty)}|T(b)(x)|d\gamma_\alpha(x)+\int_{(0,\infty)\setminus (2I)}|T(b)(x)|d\gamma_\alpha(x)\\
&=A_1+A_2.
\end{align*}
Since $T$ is bounded on $L^2((0,\infty),\gamma_\alpha)$, by using~\ref{norm-b} we get
\[A_1\le \|T(b)\|_{L^2((0,\infty),\gamma_\alpha)}(\gamma_\alpha(2I))^{1/2}\le C\|b\|_{L^2((0,\infty),\gamma_\alpha)}(\gamma_\alpha(I))^{1/2}\le C.\]
According to~\ref{sop-b} and~\ref{int-b}, the properties of (II)-type operators lead to
\begin{align*}
A_2&=\int_{(2I)^c}\left\|\int_Ib(y)(K_t(x,y)-K_t(x,x_I))d\gamma_\alpha(y)\right\|_Xd\gamma_\alpha(x)\\
&\le \int_I|b(y)|\int_{(2I)^c}\|K_t(x,y)-K_t(x,x_I)\|_Xd\gamma_\alpha(x)d\gamma_\alpha(y)\\
&\le C\int_I|b(y)|d\gamma_\alpha(y)\le C\|b\|_{L^2((0,\infty, \gamma_\alpha))}(\gamma_\alpha(I))^{1/2}\le C.
\end{align*}
We conclude that
\[
\|T(b)\|_{L^1((0,\infty),\gamma_\alpha)}\le C,
\]
where $C>0$ does not depend on $b$.

By proceeding as in the end of the proof of \cite[Theorem~1.2]{BDQS-2} we can finish this one. Indeed, suppose that $f=\sum_{j=0}^\infty\lambda_jb_j$, where, for every $j\in \mathbb{N}$, $b_j$ is a $(1,2,\alpha)$--atom and $\lambda_j\in \mathbb{C}$ being $\sum_{j=0}^\infty|\lambda_j|<\infty$. The series defining $f$ converges in $L^1((0,\infty),\gamma_\alpha)$. Since $T$ is bounded from $L^1((0,\infty),\gamma_\alpha)$ into $L^{1,\infty}((0,\infty),\gamma_\alpha)$ we have that
\[
T(f)=\lim_{k\to\infty}T\left(\sum_{j=0}^k\lambda_jb_j\right), \text{ in } L^{1,\infty}((0,\infty),\gamma_\alpha).
\]
Then, there exists an increasing function $\psi:\mathbb{N}\rightarrow \mathbb{N}$ such that
\[
T(f)(x)=\lim_{k\to\infty}T\left(\sum_{j=0}^{\psi(k)}\lambda_jb_j\right)(x),\quad \text{a.e. } x\in (0,\infty),
\]
and thus, for almost every $x\in (0,\infty),$
\[
|T(f)(x)|=\lim_{k\to\infty}T\left(\sum_{j=0}^{\psi(k)}\lambda_jb_j\right)(x)\le\lim_{k\to\infty}\sum_{j=0}^{\psi(k)}|\lambda_j|T(b_j)(x)=\sum_{j=0}^{\infty}|\lambda_j|T(b_j)(x).
\]
It follows that
\[
\|T(f)\|_{L^1((0,\infty),\gamma_\alpha)}\le \sum_{j=0}^\infty|\lambda_j|\|T(b_j)\|_{L^1((0,\infty),\gamma_\alpha)}\le C\sum_{j=0}^\infty |\lambda_j|.
\]
We conclude that $\|T(f)\|_{L^1((0,\infty),\gamma_\alpha)}\le C\|f\|_{H^1((0,\infty),\gamma_\alpha)}$ and the proof is now finished.

\begin{rem} Some comments about the proof of Theorem~\ref{criterion}~\ref{H1-L1} are in order. Under the conditions given in Theorem~\ref{criterion}~\ref{H1-L1} we prove that there exists $C>0$ such that $\|T(b)\|_{L^1((0,\infty),\gamma_\alpha)}\le C$, for every $(1,2,\alpha)$--atom $b$. We need the operator $T$ to be bounded from $L^1((0,\infty),\gamma_\alpha)$ into $L^{1,\infty}((0,\infty),\gamma_\alpha)$ because the results in \cite{MSV} and \cite{YZ} cannot be applied. Note that our underlying space is not of homogeneous type.
\end{rem}

\section{Some technical lemmas}\label{lemmas-aux}

In this section we list and prove auxiliary results which are the key ingredients for the proof of Theorem~\ref{endpoints}.

We begin by establishing some estimates involving the integral kernel of the heat semigroup $\{W_t^\alpha\}_{t>0}$. The integral representation of the modified Bessel function of the first kind $I_\nu$, $\nu>-\frac12$ (see \cite[(5.10.22)]{Leb}), leads to
\[W_t^\alpha(x,y)=\frac{1}{(1-e^{-t})^{\alpha+1}}\int_{-1}^1 \exp\left(-\frac{q\left(e^{-t/2}x,y,s\right)}{1-e^{-t}}+y^2\right)\Pi_\alpha(s)ds,\]
for $x,y,t\in (0,\infty)$, where $\Pi_\alpha(s):=\frac{\Gamma(\alpha+1)}{\Gamma(\alpha+1/2)\sqrt{\pi}}(1-s^2)^{\alpha-1/2}$ for $s\in (-1,1)$ and $q(x,y,s):=x^2+y^2-2xys$ for $x,y\in (0,\infty)$ and $s\in (-1,1)$.

For every $n\in \mathbb N$, we denote by $H_n$ the Hermite polynomial of degree $n$ given by 
\begin{equation}\label{Hpoly}
    H_n(x)=(-1)^n e^{x^2}\frac{d^n}{dx^n}\left(e^{-x^2}\right),\quad  x\in \mathbb R.
\end{equation}
Notice that for any $n\in \mathbb N$, $x,y,t\in (0,\infty)$ and $s\in (-1,1)$ we have that
\begin{align*}
    \frac{d^n}{dx^n}\exp\left(-\frac{q\left(e^{-t/2}x,y,s\right)}{1-e^{-t}}\right)&=\frac{d^n}{dx^n}\exp\left(-\left(\frac{e^{-t/2}x-ys}{\sqrt{1-e^{-t}}}\right)^2\right)e^{-\frac{1-s^2}{1-e^{-t}}y^2}\\
    &=\left(\frac{e^{-t/2}}{\sqrt{1-e^{-t}}}\right)^n \left.\frac{d^n}{dz^n} e^{-z^2}\right|_{z=\frac{e^{-t/2}x-ys}{\sqrt{1-e^{-t}}}}e^{-\frac{1-s^2}{1-e^{-t}}y^2}\\
    &=\left(-\frac{e^{-t/2}}{\sqrt{1-e^{-t}}}\right)^n H_n\left(\frac{e^{-t/2}x-ys}{\sqrt{1-e^{-t}}}\right)e^{-\frac{q\left(e^{-t/2}x,y,s\right)}{1-e^{-t}}}.
\end{align*}
Then, given $x,y,t\in (0,\infty)$,
\begin{align*}
  \partial_x^n W_t^\alpha (x,y)&=\frac{1}{(1-e^{-t})^{\alpha+1}}\left(-\frac{e^{-t/2}}{\sqrt{1-e^{-t}}}\right)^n \\
  &\quad \times \int_{-1}^1 H_n\left(\frac{e^{-t/2}x-ys}{\sqrt{1-e^{-t}}}\right)e^{-\frac{q\left(e^{-t/2}x,y,s\right)}{1-e^{-t}}+y^2} \Pi_\alpha(s)ds,  
\end{align*}
and we can also see that
\begin{align*}
  \partial_y\partial_x W_t^\alpha (x,y)&=\frac{2e^{-t/2}}{(1-e^{-t})^{\alpha+3/2}} \int_{-1}^1 e^{-\frac{q\left(e^{-t/2}x,y,s\right)}{1-e^{-t}}+y^2} \\
  &\quad \times   \left(\frac{s}{\sqrt{1-e^{-t}}}+\frac{2e^{-t/2}\left(e^{-t/2}y-xs\right)\left(e^{-t/2}x-ys\right)}{(1-e^{-t})^{3/2}}\right)\Pi_\alpha(s)ds.
\end{align*}

From all of the above and considering, as usual, the change of variables $r=e^{-t/2}$, $t\in (0,\infty)$, we deduce the following result.

\begin{lem}
\label{lem3.7}
    Let $\alpha>-\frac12$ and $n\in \mathbb N$. There exists $C>0$ such that for any $x,y\in (0,\infty)$ and $r\in (0,1)$, 
    \begin{align*}
    \left|\partial_x^n W_{-2\log r}^\alpha (x,y)\right|&\leq C r^n\int_{-1}^1 \left(1+\left|\frac{rx-ys}{\sqrt{1-r^2}}\right|^n\right)\frac{e^{-\frac{q(rx,y,s)}{1-r^2}+y^2}}{(1-r^2)^{\alpha+1+\frac n2}} \Pi_\alpha(s)ds,
    \end{align*}
    and
    \begin{align*}
    \left|\partial_y\partial_x W_{-2\log r}^\alpha (x,y)\right|&\leq C r\int_{-1}^1 \left(\frac{|s|}{\sqrt{1-r^2}}+\frac{2r|ry-xs||rx-ys|}{(1-r^2)^{3/2}}\right)\\
    &\quad \times\frac{e^{-\frac{q(rx,y,s)}{1-r^2}+y^2}}{(1-r^2)^{\alpha+\frac 32}} \Pi_\alpha(s)ds.
    \end{align*}
\end{lem}

We now state two lemmas that will be useful for several estimates we will give below. The first lemma was given in \cite[Lemma~4]{BCCFR} when $\sigma\in \mathbb N$, $\sigma\geq 1$. We include here the case~${\sigma=0}$ since it also holds, as the reader can check. The second one deals with the derivatives of the Poisson kernel, involved in the definition of many of the operators considered in this article.

\begin{lem} \label{Lema4BCCFR} Let $\sigma\in \mathbb{N}$. Then
\[\left|\partial_t^\sigma \left(t e^{-\frac{t^2}{4u}}\right)\right|\leq C e^{-\frac{t^2}{8u}} u^{\frac{1-\sigma}{2}}, \quad t,u\in (0,\infty).\]
\end{lem}

\begin{lem}\label{cotasderivadasPoisson}
    Let $n,k\in \mathbb{N}$ and $\ell=0,1$. Then, there exists $C>0$ such that for every $x,y,t\in (0,\infty)$,
    \begin{equation*}
        \left|t^{n+k} \partial_y^\ell \partial_x^{n+1-\ell}\partial_t^k P^\alpha_t (x,y)\right|\leq C\int_0^1 \left|\partial_y^\ell \partial_x^{n+1-\ell} W_{-2\log r}^\alpha(x,y)\right| \frac{t^{n+k} e^{t^2/(16\log r)}}{(-\log r)^{\frac k2+1}} \frac{dr}{r}.
    \end{equation*}
\end{lem}

\begin{proof}The proof is immediate by the subordination formula and Lemma~\ref{Lema4BCCFR} with $u=-2\log r$ and $\sigma=k$, which give
\begin{align*}
&\left|t^{n+k} \partial_y^\ell \partial_x^{n+1-\ell}\partial_t^k P^\alpha_t (x,y)\right|\\
&\quad \leq C\int_0^1 \left|\partial_y^\ell \partial_x^{n+1-\ell} W_{-2\log r}^\alpha(x,y)\right| t^{n+k}\left|  \partial_t^k\left(t e^{t^2/(8\log r)}\right)\right| \frac{dr}{r(-\log r)^{\frac32}}\\
&\quad \leq C\int_0^1 \left|\partial_y^\ell \partial_x^{n+1-\ell} W_{-2\log r}^\alpha(x,y)\right| \frac{t^{n+k} e^{t^2/(16\log r)}}{(-\log r)^{\frac k2+1}} \frac{dr}{r}. \qedhere
\end{align*}
\end{proof}

The following lemma establishes some properties for the function $q$, which are all straightforward from its definition.

\begin{lem}\label{several-ineq} 
Let $x,y\in (0,\infty)$, $r\in (0,1)$ and $s\in (-1,1)$. Then, the following estimates hold.
\begin{enumerate}[label=(E\arabic*), start=0]
    \item \label{eq-D0} $\frac{q(rx,y,s)}{1-r^2}-y^2=\frac{q(x,ry,s)}{1-r^2}-x^2$;
    \item \label{eq-D1} $q(x,ry,s) = 
    (x-ry)^2 + 2xyr(1-s)$ and so
    \item \label{eq-D5} $q(x,ry,s)\geq 2xyr(1-s)$;
    \item \label{eq-D3} $q(x,ry,s)
    \geq r^2(x^2+y^2)$ when $s\in (-1,0)$ and
    \item \label{eq-D4} $q(x,ry,s)
    \geq r^2(x^2+y^2)(1-s)$ when $s\in [0,1)$;
    \item \label{eq-D7} $q(x,ry,s) \geq x^2r^2(1-s)$;
    \item \label{eq-D8} $q(x,ry,s) \geq y^2r^2(1-s)$;
    \item \label{eq-D9} $q(x,ry,s) \geq (x-rys)^2+y^2r^2(1-s^2)$;
    \item \label{eq-D6} $q(rx,y,s)
    \geq (rx-ys)^2$.
    \end{enumerate}
\end{lem}

Many of the estimates on the kernels of the operators considered here will require the boundedness properties of the functions $\varphi$, $\psi$ and $\xi_n$, for $n\in \mathbb N$, defined in the next lemma.

\begin{lem}\label{func-acot} Let $n\in \mathbb N$. The functions
\[\varphi(r)=\frac{1-r^2}{-\log r}, \quad \psi(r)=\frac{r(-\log r)}{1-r^2}, \quad \xi_n(r)=\frac{r^n(-\log r)^{\frac n2-1}}{(1-r^2)^{\frac n2-1}},\]
with $\varphi(0)=\psi(0)=\xi_n(0)=0$, are bounded on $[0,1]$.
\end{lem}

In what follows, we establish several estimates that will led us to obtain conditions \eqref{C1} and \eqref{C2} in order to prove Theorem~\ref{endpoints} in Section~\ref{mainthm}.

In what follows, given $I\in \mathcal B_1$, we denote with $c_I$ and $r_I$ the center and radius of $I$, respectively.

\begin{lem}
\label{lem3.1} Let $\alpha>-\frac12$. There exists $C>0$ such that, for any interval $I\in \mathcal B_1$, 
\[\sup_{y\in I} r_I \int_{(2I)^c} K(x,y) d\gamma_\alpha(x)\leq C\]
where
\[K(x,y)=\int_0^1 \int_{-1}^1 \frac{e^{-\frac{q(x,ry,s)}{2(1-r^2)}+x^2}}{(1-r^2)^{\alpha+5/2}}\Pi_\alpha(s) dsdr.\]
\end{lem}
\begin{proof}
Assume that $I\in \mathcal{B}_1$ and $y\in I$. For every $r\in \left(1-\frac{r_I}{2(1+y)},1\right)$ and $x\notin 2I$, by~\ref{eq-D1} we get
\begin{align}\label{eq-D2}
    \sqrt{q(x,r y, s)} &\geq |x-ry| \geq |x-y| - (1-r)y \geq |x-y| -  r_I\frac{ y}{2(1+y)} \nonumber\\
    &\geq |x-y| - \frac{r_I}{2} \geq \frac{|x-y|}{2},
\end{align}
for any $s\in(-1,1)$.

It yields
\begin{equation*}
    \begin{split}
        \int_{-1}^1 &\int_{1-\frac{r_I}{2(1+y)}}^1 
        e^{-\frac{q(x,ry,s)}{2(1-r^2)}} \frac{dr}{(1-r^2)^{\alpha + 5/2}} \Pi_{\alpha} (s) ds 
        \\ & \leq C \int_{-1}^1 \int_{1-\frac{r_I}{2(1+y)}}^1
        e^{-\frac{q(x,ry,s)}{4(1-r^2)}}
        \frac{e^{-\frac{|x-y|^2}{16(1-r^2)}}}{(1-r^2)^{\alpha + 5/2}} dr\Pi_{\alpha} (s) ds 
        \\ & \leq C
        \int_{1-\frac{r_I}{2(1+y)}}^1 \frac{e^{-\frac{|x-y|^2}{16(1-r^2)}}}{(1-r^2)^{3/2}} \int_{-1}^1 \frac{\Pi_\alpha(s)}{q(x,ry,s)^{\alpha +1}} ds dr,
    \end{split}
\end{equation*}
for $x\notin 2I$. According to~\cite[Lemma~2.1]{NSz}  and \eqref{eq-D2} we obtain
\begin{equation*}
    \begin{split}
        \int_{-1}^1 \int_{1-\frac{r_I}{2(1+y)}}^1 &
         \frac{e^{-\frac{q(x,ry,s)}{2(1-r^2)}}}{(1-r^2)^{\alpha + 5/2}} dr\Pi_{\alpha} (s) ds 
        \\ & \leq C 
        \int_{1-\frac{r_I}{2(1+y)}}^1
        \frac{e^{-\frac{|x-y|^2}{16(1-r^2)}}}{(1-r^2)^{3/2}} \frac{dr}{\mathfrak{m}_\alpha(I(x,|x-ry|))}
        \\ & \leq C
        \int_{1-\frac{r_I}{2(1+y)}}^1
        \frac{e^{-\frac{|x-y|^2}{16(1-r^2)}}}{(1-r^2)^{3/2}} \frac{dr}{\mathfrak{m}_\alpha(I(x,|x-y|/2))},
    \end{split}
\end{equation*}
for $x\notin 2I$. Here, $\mathfrak{m}_\alpha$ denotes the measure defined by $d\mathfrak{m}_\alpha(x) = x^{2\alpha +1} dx$ in $(0,\infty)$.

Also, we can write
\[\int_{1-\frac{r_I}{2(1+y)}}^1
        \frac{e^{-\frac{|x-y|^2}{16(1-r^2)}}}{(1-r^2)^{3/2}} dr \leq
        C \int_0^{\infty} \frac{e^{-\frac{|x-y|^2}{u}}}{u^{3/2}} du
        \leq \frac{C}{|x-y|},\]
for $x\notin 2I$. Then,
\begin{equation*}
\begin{split}
     \int_{-1}^1 \int_{1-\frac{r_I}{2(1+y)}}^1 &
         \frac{e^{-\frac{q(x,ry,s)}{2(1-r^2)}}}{(1-r^2)^{\alpha + 5/2}} dr\Pi_{\alpha} (s) ds\leq \frac{C}{|x-y|\mathfrak{m}_\alpha(I(x,|x-y|/2))}
\end{split}
\end{equation*}
for any $x\notin 2I$.

By using~\cite[Lemma~2.2]{NSz} and~\ref{eq-D0},
\begin{equation*}
\begin{split}
    \int_{(2I)^c}
     \int_{-1}^1 \int_{1-\frac{r_I}{2(1+y)}}^1 &
         \frac{e^{-\frac{q(x,ry,s)}{2(1-r^2)}+x^2}}{(1-r^2)^{\alpha + 5/2}} dr\Pi_{\alpha} (s) ds d\gamma_\alpha(x)
        \\ & \leq C
        \int_{(2I)^c} 
        \frac{x^{2\alpha +1}}{|x-y|^{2}\left(
        x + \frac{|x-y|}{2}\right)^{2\alpha +1}}dx
        \\ & \leq C
        \int_{(2I)^c} \frac{1}{|x-y|^2} dx
        \\ & \leq C
        \int_{r_I}^\infty \frac{1}{z^2} dz
        \\ & \leq \frac{C}{r_I}.
\end{split}
\end{equation*}
Here $C>0$ does not depend on $I$ or $y\in I$.

To deal with the remaining term we consider, for every $x\in (0,\infty)$, the Hankel translation defined by (see~\cite{GoS})
\begin{equation*}
    {}_{\alpha}\tau_x (f)(y) = 
    \int_{-1}^1 f\left(\sqrt{x^2 + y^2 - 2xys}\right) \Pi_\alpha(s)ds, \;\;x,\,y\in(0,\infty),
\end{equation*}
which is contractive in $L^p((0,\infty), \mathfrak{m}_\alpha)$ for every $x\in (0,\infty)$ and $1\leq p\leq \infty$ (\cite[p.~657]{GoS}). 

We deduce that, for every $y\in I$, 
\begin{equation*}
\begin{split}
    &\int_{(2I)^c}
     \int_{-1}^1 \int_0^{1-\frac{r_I}{2(1+y)}} 
         \frac{e^{-\frac{q(x,ry,s)}{2(1-r^2)}+x^2}}{(1-r^2)^{\alpha + 5/2}}dr \Pi_{\alpha} (s) ds d\gamma_\alpha(x)
        \\ & \leq 
        \int_0^{1-\frac{r_I}{2(1+y)}}
        \frac{1}{(1-r^2)^{\alpha + 5/2}}  \int_{0}^\infty
     \int_{-1}^1
     e^{-\frac{q(x,ry,s)}{2(1-r^2)}} \Pi_{\alpha} (s) ds d\mathfrak{m}_\alpha(x) dr
     \\ & =
     \int_0^{1-\frac{r_I}{2(1+y)}}
        \frac{1}{(1-r^2)^{\alpha + 5/2}}  \int_{0}^\infty
        {}_{\alpha}{\tau}_{ry}
        \left(e^{-\frac{z^2}{2(1-r^2)}}\right)(x)d\mathfrak{m}_\alpha(x) dr
         \\ & \leq
     \int_0^{1-\frac{r_I}{2(1+y)}}
        \frac{1}{(1-r^2)^{\alpha + 5/2}}  \int_{0}^\infty
         e^{-\frac{z^2}{2(1-r^2)}}d\mathfrak{m}_\alpha(z) dr
         \\ & \leq
     \int_0^{1-\frac{r_I}{2(1+y)}}
        \frac{dr}{(1-r^2)^{3/2}} 
         \\ & \leq
     \int_{\frac{r_I}{2(1+y)}}^\infty
        \frac{dr}{r^{3/2}} 
         \\ & \leq
     C\left(\frac{1+y}{r_I}\right)^{1/2}.
\end{split}
\end{equation*}

By combining the above estimates we get
\begin{equation*}
    \begin{split}
        \sup_{y\in I} r_I \int_{(2I)^c} K(x,y) d\gamma_\alpha (x) & 
        \leq C \sup_{y\in I} 
        \left(1 + (r_I y)^{1/2}\right)
        \\ & \leq C \left(
        1+ (r_I(r_I+c_I))^{1/2}
        \right)
        \\ & \leq C,
    \end{split}
\end{equation*}
where $C>0$ does not depend on $I$.
\end{proof}

\begin{lem}
\label{lem3.1.1}
    Let $\alpha>-\frac12$ and $\omega>0$. Then there exists $C>0$ such that, for every $I\in \mathcal{B}_1$,
    \begin{equation*}
        \sup_{y\in I}r_I\int_{(2I)^c} K_\omega(x,y) d\gamma_\alpha(x)\leq C,
    \end{equation*}
    where for $x,y\in(0,\infty)$ and $x\neq y$.
    \begin{equation*}
        K_\omega(x,y) =  \int_{-1}^1 \int_0^1 
       e^{-\frac{q(x,ry,s)}{2(1-r^2)}+x^2}
        \frac{(-\log r)^{\omega-1}}{(1-r^2)^{\alpha + 3/2}}
        \Pi_\alpha(s)dsdr.
    \end{equation*}
\end{lem}

\begin{proof}
    We proceed as in the proof of Lemma~\ref{lem3.1}. Let $I\in \mathcal{B}_1$. Since \[1-\frac{r_I}{2(1+y)}\geq \frac12,\quad y\in (0,\infty),\] and using the boundedness of the function $\varphi$ given in Lemma~\ref{func-acot}, it follows that
    \begin{equation*}
        \begin{split}
            K_{\omega,1}(x,y) & = 
            \int_{-1}^1 \int_{1-\frac{r_I}{2(1+y)}}^1
        e^{-\frac{q(x,ry,s)}{2(1-r^2)}+x^2}
        \frac{(-\log r)^{\omega-1}}{(1-r^2)^{\alpha + 3/2}}
        \Pi_\alpha(s)drds
        \\ & C\leq 
        \int_{-1}^1 \int_{1-\frac{r_I}{2(1+y)}}^1
        e^{-\frac{q(x,ry,s)}{2(1-r^2)}+x^2}
        \frac{ \Pi_\alpha(s)}{(1-r^2)^{\alpha + 5/2 - \omega}}
       drds
       \\ & \leq 
        \int_{-1}^1 \int_{1-\frac{r_I}{2(1+y)}}^1
        e^{-\frac{q(x,ry,s)}{2(1-r^2)}+x^2}
        \frac{ \Pi_\alpha(s)}{(1-r^2)^{\alpha + 5/2}}
       drds
       \\ & \leq 
        \int_{-1}^1 \int_{0}^1
        e^{-\frac{q(x,ry,s)}{2(1-r^2)}+x^2}
        \frac{ \Pi_\alpha(s)}{(1-r^2)^{\alpha + 5/2}}
        drds,
        \end{split}
    \end{equation*}
for $x,y\in(0,\infty)$ and $x\neq y$. Lemma~\ref{lem3.1} implies that
\begin{equation*}
        \sup_{y\in I}r_I\int_{(2I)^c} K_{\omega,1}(x,y) d\gamma_\alpha(x)\leq C,
    \end{equation*}
    where $C>0$ does not depend on $I$. 

    On the other hand, we consider, for $x,y\in(0,\infty)$ and $x\neq y$, 
    \begin{equation*}
        K_{\omega,2}(x,y) =  \int_{-1}^1 \int_0^{1-\frac{r_I}{2(1+y)}}
        e^{-\frac{q(x,ry,s)}{2(1-r^2)}+x^2}
        \frac{(-\log r)^{\omega-1}}{(1-r^2)^{\alpha + 3/2}}
        \Pi_\alpha(s)drds.
    \end{equation*}
    Suppose $\omega\in (0,1]$. From Lemma~\ref{func-acot} we have that
    \begin{equation*}
        K_{\omega,2}(x,y) \leq C \int_{-1}^1 \int_0^{1-\frac{r_I}{2(1+y)}}
        e^{-\frac{q(x,ry,s)}{2(1-r^2)}+x^2}
        \frac{\Pi_\alpha(s)}{(1-r^2)^{\alpha + 5/2}}
       dr ds.
    \end{equation*}
for $x,y\in(0,\infty)$ and $x\neq y$. By using Lemma~\ref{lem3.1} we deduce that
\begin{equation*}
        \sup_{y\in I}r_I\int_{(2I)^c} K_{\omega,2}(x,y) d\gamma_\alpha(x)\leq C,
    \end{equation*}
   for certain $C>0$ that does not depend on $I$, provided that $\omega \in (0,1]$. 

   Assume now that $\omega>1$. As in the proof of Lemma~\ref{lem3.1}, and applying again Lemma~\ref{func-acot}, we get, for every $y\in I$, 
\begin{equation*}
\begin{split}
     \int_{(2I)^c} K_{\omega,2}(x,y) d\gamma_\alpha(x)
     & \leq 
     C \int_0^{1-\frac{r_I}{2(1+y)}} \frac{(-\log r)^{\omega-1}}{(1-r)^{1/2}}dr
     \\ & \leq C
     \int_0^{1} \frac{(-\log r)^{\omega-1}}{(1-r)^{1/2}}dr \leq C.
\end{split}
    \end{equation*} 
   Then, for any $\omega >0$, 
   \begin{equation*}
        \sup_{y\in I}r_I\int_{(2I)^c} K_{\omega,2}(x,y) d\gamma_\alpha(x)\leq C,
    \end{equation*}
    where $C>0$ does not depend on $I$. 
\end{proof}

\begin{lem}
\label{lem3.2}
    Let $\alpha>-\frac12$. There exists $C>0$ such that, for every $I\in \mathcal{B}_1$,
    \begin{equation*}
        \sup_{y\in I}r_I\int_{(2I)^c} |K(x,y)| d\gamma_\alpha(x)\leq C,
    \end{equation*}
    where for $x,y\in(0,\infty)$,
    \begin{equation*}
        K(x,y) =  \int_{-1}^1 \int_0^1 \varphi(r)^{1/2}
        \frac{\partial}{\partial r} \left[e^{-\frac{q(x,ry,s)}{1-r^2}+x^2}\right]
        \frac{r dr}{(1-r^2)^{\alpha + 3/2}}
        \Pi_\alpha(s)ds.
    \end{equation*}
\end{lem}

\begin{proof}
    Integrating by parts we get, for $x,y\in (0,\infty)$,
    \begin{equation*}
        \begin{split}
            K(x,y) & = \int_{-1}^1 \left( \left[
                     \frac{\varphi(r)^{1/2} r}{(1-r^2)^{\alpha + 3/2}} 
            e^{-\frac{q(x,ry,s)}{1-r^2}+x^2}
            \right]_{r=0}^{r=1} \right.
            \\ & \quad - 
            \int_0^1 \left. \frac{\partial}{\partial r}
            \left(
            \frac{\varphi(r)^{1/2} r}{(1-r^2)^{\alpha + 3/2}} \right) 
            e^{-\frac{q(x,ry,s)}{1-r^2}+x^2}
            \right) \Pi_\alpha (s) ds
            \\ & = - 
            \int_{-1}^1 \int_0^1 
            \frac{\partial}{\partial r} \left(
            \frac{\varphi(r)^{1/2}r}{(1-r^2)^{\alpha + 3/2}} \right) 
            e^{-\frac{q(x,ry,s)}{1-r^2}+x^2}
            dr \Pi_\alpha (s) ds.
        \end{split}
    \end{equation*}
    By virtue of Lemma~\ref{func-acot}, for $r\in(0,1)$, we get
    \begin{equation*}
        \left| \frac{\partial}{\partial r} \left(
            \frac{\varphi(r)^{1/2}r}{(1-r^2)^{\alpha + 3/2}} \right) 
        \right| \leq \frac{C}{(1-r^2)^{\alpha+5/2}}.
    \end{equation*}
    It follows that, for $x,y\in(0,\infty)$,
    \begin{equation*}
        |K(x,y)| \leq C 
        \int_{-1}^1 \int_0^1 
        \frac{ e^{-\frac{q(x,ry,s)}{1-r^2}+x^2}}{(1-r^2)^{\alpha+5/2}}  dr \Pi_\alpha (s) ds.
    \end{equation*}
    Thus, Lemma~\ref{lem3.1} allows us to finish the proof.
\end{proof}

\begin{lem}
\label{lem3.3}
    Let $\alpha> -\frac12$. There exists $C>0$ such that, for every $I\in \mathcal{B}_1$,
    \[\sup_{y\in I} r_I \int_{(2I)^c} |K(x,y)| d\gamma_\alpha(x)\leq C,\]
    where, for $x,y\in (0,\infty)$, 
    \[K(x,y)=\int_{-1}^1 \int_0^1
    \frac{r(rx-ys)(ry-xs)}{(1-r^2)^{\alpha+7/2}}e^{-\frac{q(x,ry,s)}{1-r^2}+x^2}dr\Pi_\alpha(s)ds.\]
\end{lem}

\begin{proof}
First, let us notice that, for every $x,y\in (0,\infty)$, $r\in (0,1)$ and $s\in (-1,1)$,
\begin{align*}
    \frac{\partial}{\partial r}& \left[e^{-\frac{q(x,ry,s)}{1-r^2}}\right]\\
    &=-\frac{(2ry^2-2xys)(1-r^2)+2r(x^2+r^2y^2-2xyrs)}{(1-r^2)^2}
    e^{-\frac{q(x,ry,s)}{1-r^2}}\\
    &=-\frac{2r(x^2+y^2)-2xys(1+r^2)}{(1-r^2)^2}e^{-\frac{q(x,ry,s)}{1-r^2}}
\end{align*}
and 
\[(rx-ys)(ry-xs)=(r^2+s^2)xy-rs(x^2+y^2).\]
Thus, for every $x,y\in (0,\infty)$, $r\in (0,1)$ and $s\in (-1,1)$, we can write
\begin{align*}
    \frac{2(rx-ys)(ry-xs)}{(1-r^2)^2}&=\frac{\partial}{\partial r} \left[e^{-\frac{q(x,ry,s)}{1-r^2}}\right]e^{\frac{q(x,ry,s)}{1-r^2}}+\frac{2r(x^2+y^2)(1-s)}{(1-r^2)^2}\\
    &\quad +\frac{2xy(1-s)(r^2-s)}{(1-r^2)^2}.
\end{align*}
Taking the above expression into account, we define, for every $x,y\in (0,\infty)$,
\begin{align*}
    K_1(x,y)&=\frac12 e^{x^2} \int_{-1}^1 \int_0^1 
    \frac{r}{(1-r^2)^{\alpha+3/2}} \frac{\partial}{\partial r} \left[e^{-\frac{q(x,ry,s)}{1-r^2}}\right]
dr\Pi_\alpha(s)ds,\\
    K_2(x,y)&=e^{x^2}\int_{-1}^1 \int_0^1 
     \frac{r^2(x^2+y^2)(1-s)}{(1-r^2)^{\alpha+7/2}} e^{-\frac{q(x,ry,s)}{1-r^2}}
dr\Pi_\alpha(s)ds,\\
    K_3(x,y)&=e^{x^2}\int_{-1}^1 \int_0^1 
     \frac{xyr(1-s)(r^2-s)}{(1-r^2)^{\alpha+7/2}} e^{-\frac{q(x,ry,s)}{1-r^2}}
dr\Pi_\alpha(s)ds,
\end{align*}
which clearly verify
\[K(x,y)=K_1(x,y)+K_2(x,y)+K_3(x,y), \quad x,y\in (0,\infty).\]

As Lemma~\ref{lem3.2} takes care of $K_1(x,y)$, it only remains to find estimates for $K_2$ and $K_3$. We shall actually see that both of them are bounded by 
\[A(x,y):=\int_0^1 \int_{-1}^1 \frac{e^{-\frac{q(x,ry,s)}{2(1-r^2)}+x^2}}{(1-r^2)^{\alpha+5/2}}\Pi_\alpha(s) dsdr,\quad x,y\in (0,\infty),\]
which verifies the desired inequality as proved in Lemma~\ref{lem3.1}.

Indeed, for $K_2$ we use~\ref{eq-D3},~\ref{eq-D4} and the fact that $\varphi$ is bounded on $(0,1)$ (see Lemma~\ref{func-acot}) to get, for every $x,y\in (0,\infty)$,
\[0\leq K_2(x,y)\leq C\int_{-1}^1 \int_0^1 \frac{q(x,ry,s)}{(1-r^2)^{\alpha+7/2}} e^{-\frac{q(x,ry,s)}{1-r^2}+x^2}
dr\Pi_\alpha(s)ds\leq CA(x,y).\]
For $K_3$, we again apply Lemma~\ref{func-acot}, together with inequality~\ref{eq-D5}, to have
\[|K_3(x,y)|\leq C\int_{-1}^1 \int_0^1 \frac{q(x,ry,s)}{(1-r^2)^{\alpha+7/2}} e^{-\frac{q(x,ry,s)}{1-r^2}+x^2}
dr\Pi_\alpha(s)ds\leq CA(x,y).\]
This finishes the proof.
\end{proof}

\begin{lem}
\label{lem3.4}
  Let $\alpha> -\frac12$ and $\beta\geq 0$. There exists $C>0$ such that, for every $I\in \mathcal{B}_1$,
    \[\sup_{x\in I} r_I \int_{(2I)^c} K_\beta (x,y) d\gamma_\alpha(y)\leq C,\]
    where, for $x,y\in (0,\infty)$,
    \[K_\beta(x,y)=\int_{-1}^1 \int_0^1 \left|\frac{rx-ys}{\sqrt{1-r^2}}\right|^\beta \frac{e^{-\frac{q(rx,y,s)}{1-r^2}+y^2}}{(1-r^2)^{\alpha+5/2}}dr\Pi_\alpha(s)ds.\]
\end{lem}

\begin{proof}According to~\ref{eq-D6}, $|rx-ys|^\beta \leq (q(rx,y,s))^{\beta/2}$ for any $x,y\in (0,\infty)$, $r\in(0,1)$ and $s\in (-1,1).$ Hence,
\begin{align*}
  K_\beta(x,y)&\leq C\int_{-1}^1 \int_0^1 \left(\frac{q(rx,y,s)}{1-r^2}\right)^{\beta/2} \frac{e^{-\frac{q(rx,y,s)}{1-r^2}+y^2}}{(1-r^2)^{\alpha+5/2}}dr\Pi_\alpha(s)ds\\
  &\leq C\int_{-1}^1 \int_0^1 \frac{e^{-\frac{q(rx,y,s)}{2(1-r^2)}+y^2}}{(1-r^2)^{\alpha+5/2}}dr\Pi_\alpha(s)ds,\quad x,y\in (0,\infty).
\end{align*}
Since $q(rx,y,s)=q(y,rx,s)$, by changing the roles of $x$ and $y$ in Lemma~\ref{lem3.1}, the proof in finished.
\end{proof}

\begin{lem}
\label{lem3.5}
  Let $\alpha> -\frac12$. There exists $C>0$ such that, for every $I\in \mathcal{B}_1$,
    \[\sup_{x\in I} r_I \int_{(2I)^c} K(x,y) d\gamma_\alpha(y)\leq C,\]
    where, for $x,y\in (0,\infty)$,
    \[K(x,y)=\int_{-1}^1 \int_0^1 \frac{xyr}{(1-r^2)^{\alpha+7/2}} e^{-\frac{q(rx,y,s)}{1-r^2}+y^2}dr\Pi_{\alpha+1}(s)ds.\]
\end{lem}
\begin{proof}
From~\ref{eq-D1} we know that $q(rx,y,s)=q(y,rx,s)\geq 2xyr(1-s)$ for $x,y\in (0,\infty)$, $r\in(0,1)$ and $s\in (-1,1).$ Since $\Pi_{\alpha+1}(s)\leq C_\alpha \Pi_\alpha(s) 2(1-s)$ for each $s\in (-1,1)$, we have 
\[K(x,y)\leq C \int_{-1}^1 \int_0^1 \frac{e^{-\frac{q(rx,y,s)}{1-r^2}+y^2}}{(1-r^2)^{\alpha+5/2}} dr\Pi_\alpha(s)ds,\quad x,y\in (0,\infty).\]
As in the proof of the previous lemma, the desired inequality for $K$ follows from Lemma~\ref{lem3.1}.
\end{proof}

\begin{lem}
\label{lem3.6}
    Let $\alpha>-\frac12$. There exists $C>0$ such that, for every $I\in\mathcal{B}_1$,
    \begin{equation*}
    \sup_{y\in I}r_I \int_{(2I)^c} |K_j(x,y)| d\gamma_\alpha(x)\leq C, \quad j=1,2,3,4,
    \end{equation*}
    where, for $x,y\in(0,\infty)$,
    \begin{equation*}
        K_1(x,y) = \int_{-1}^1 \int_0^1
        \frac{e^{-\frac{q(x,ry,s)}{1-r^2}+x^2}}{(1-r^2)^{\alpha+5/2}} 
        xrdr \Pi_{\alpha+1}(s)ds,
    \end{equation*}
    \begin{equation*}
        K_2(x,y) = \int_{-1}^1 \int_0^1
        \frac{e^{-\frac{q(x,ry,s)}{1-r^2}+x^2}}{(1-r^2)^{\alpha+7/2}} 
        xr(x-ysr)dr \Pi_{\alpha+1}(s)ds,
    \end{equation*}
    \begin{equation*}
        K_3(x,y) = \int_{-1}^1 \int_0^1
        \frac{e^{-\frac{q(x,ry,s)}{1-r^2}+x^2}}{(1-r^2)^{\alpha+9/2}} 
        xyr^2(x-ysr)(ry-xs)dr \Pi_{\alpha+1}(s)ds,
    \end{equation*}
    and 
     \begin{equation*}
        K_4(x,y) = \int_{-1}^1 \int_0^1
        \frac{e^{-\frac{q(x,ry,s)}{1-r^2}+x^2}}{(1-r^2)^{\alpha+5/2}} 
        yr(ry-xs)dr \Pi_{\alpha+1}(s)ds.
    \end{equation*}
\end{lem}
\begin{proof}
     By \ref{eq-D7}, the property for $j=1$ can be deduced from Lemma~\ref{lem3.1}. 
     
     Using \ref{eq-D8} and~\ref{eq-D6}, the property for $j=4$ follows from Lemma~\ref{lem3.1}.
    
    On the other hand, by~\ref{eq-D7} and~\ref{eq-D9} we deduce that the property holds for $j=2$ after using Lemma~\ref{lem3.1}.
    
    Finally, the property for $j=3$ follows from inequalities~\ref{eq-D7},~\ref{eq-D8},~\ref{eq-D9},~\ref{eq-D6} and Lemma~\ref{lem3.1}. 
\end{proof}

\section{Proof of Theorem~\ref{endpoints}}\label{mainthm}

In this section we prove Theorem~\ref{endpoints}. We will do so by considering each operator individually and applying the auxiliary results obtained in the previous section together with the criterion given in Theorem~\ref{criterion}.

According to (\ref{subor}) we have that 
\[P_t^\alpha (f)(x)=\int_0^\infty P_t^\alpha(x,y) f(y)d\gamma_\alpha(y), \quad x\in (0,\infty),\]
where after the change of variables $u=-2\log r$, we get
\[P_t^\alpha(x,y)=\frac{1}{\sqrt{\pi}}\int_0^1 W_{-2\log r}^\alpha(x,y) t e^{t^2/(8\log r)} \frac{dr}{r(-2\log r)^{3/2}}, \quad t,x,y\in (0,\infty),\]
recalling that
\[W_{-2\log r}^\alpha(x,y)=\int_{-1}^1 \frac{e^{-\frac{q(rx,y,s)}{1-r^2}+y^2}}{(1-r^2)^{\alpha+1}}\Pi_\alpha(s) ds, \quad x,y\in (0,\infty), r\in (0,1).\]

\subsection{Proof of Theorem~\ref{endpoints} for maximal operators}\label{sec-maximal}

We will show that the operators $P_{*,k}^\alpha$ verify the properties in \ref{Linf-BMO-op} and in \ref{H1-L1-op} for any $k\in \mathbb{N}$.

Let $k\in \mathbb{N}$. According to \cite[Remark 4.3]{LeMX2} the operator $P_{*,k}^\alpha$ is bounded on $L^p((0,\infty),\gamma_\alpha)$, for every $1<p<\infty$. By using \cite[Theorem~1.1]{BdL} it follows that $P_{*,k}^\alpha$ is also bounded from $L^1((0,\infty),\gamma_\alpha)$ into $L^{1,\infty}((0,\infty),\gamma_\alpha)$.

According to Lemmas~\ref{cotasderivadasPoisson} and 
 \ref{Lema4BCCFR}, notice that
\begin{align*}\label{lem3.3-infty}
    \nonumber\left|t^{k} \partial_y^\ell \partial_x^{1-\ell}\partial_t^k P^\alpha_t (x,y)\right|&\leq C\int_0^1 \left|\partial_y^\ell \partial_x^{1-\ell} W_{-2\log r}^\alpha(x,y)\right| \frac{t^{k} e^{t^2/(16\log r)}}{(-\log r)^{\frac k2+1}} \frac{dr}{r}\\
    &\nonumber\leq C\int_0^1 \left|\partial_y^\ell \partial_x^{1-\ell} W_{-2\log r}^\alpha(x,y)\right|\frac{e^{t^2/(32\log r)}}{(-\log r)^{\frac 12}}  \frac{dr}{r}\\
    &\leq C\int_0^1 \left|\partial_y^\ell \partial_x^{1-\ell} W_{-2\log r}^\alpha(x,y)\right|(-\log r)^{-\frac 12}  \frac{dr}{r},
\end{align*}
for any $\ell=0,1$, and $x,y,t\in (0,\infty)$.

Set now $\ell=0$. By virtue of Lemma~\ref{lem3.7} and the boundedness of $\varphi$ given in Lemma~\ref{func-acot}, we get for every $x,y,t\in (0,\infty)$
\begin{align*}
    &\left|t^{k}  \partial_x\partial_t^k P^\alpha_t (x,y)\right|\\
    &\quad \leq C\int_0^1 \frac{(-\log r)^{-\frac 12}}{(1-r^2)^{\alpha+2}} \int_{-1}^1 \left(1+\left|\frac{rx-ys}{\sqrt{1-r^2}}\right|\right)e^{-\frac{q(rx,y,s)}{1-r^2}+y^2} \Pi_\alpha(s)ds dr\\
    &\quad \leq C\int_0^1 \frac{1}{(1-r^2)^{\alpha+5/2}} \int_{-1}^1 \left(1+\left|\frac{rx-ys}{\sqrt{1-r^2}}\right|\right) e^{-\frac{q(rx,y,s)}{1-r^2}+y^2}\Pi_\alpha(s) dsdr.
\end{align*}
By combining Lemmas~\ref{lem3.1} and~\ref{lem3.4}, it follows that
\begin{equation*}
    \sup_{I\in\mathcal{B}_1} r_I \sup_{x\in I} \int_{(2I)^c} \|\partial_x  t^{k}\partial_t^k P^\alpha_t (x,y)\|_{L^\infty\left((0,\infty),dt\right)} d\gamma_\alpha(y)<\infty,
\end{equation*}
meaning that $P_{*,k}^\alpha$ is bounded from $L^{\infty}((0,\infty),\gamma_\alpha)$ to $\BMO((0,\infty),\gamma_\alpha)$ and Theorem~\ref{endpoints}~\ref{Linf-BMO-op} holds.

Consider now $\ell=1$. According to \ref{eq-D0} and proceeding as in Lemma~\ref{lem3.7}, for every $x,y,t\in (0,\infty)$,
\begin{equation}\label{cotasderivadasPoisson_y}
|\partial_y W_{-2\log r}^\alpha(x,y)|\leq C\frac{r}{(1-r^2)^{\alpha+\frac32}} \int_{-1}^1 \frac{|ry-xs|}{\sqrt{1-r^2}} e^{-\frac{q(rx,y,s)}{1-r^2}+y^2} \Pi_\alpha(s)ds,
\end{equation}
and thus by using again Lemma~\ref{cotasderivadasPoisson}, we get
\[\left|t^{k} \partial_y\partial_t^k P^\alpha_t (x,y)\right|\leq C\int_{-1}^1\int_0^1 \frac{|ry-xs|}{\sqrt{1-r^2}}  \frac{e^{-\frac{q(rx,y,s)}{1-r^2}+y^2}}{(1-r^2)^{\alpha+5/2}} dr \Pi_\alpha(s)ds,\quad x,y,t\in (0,\infty).\]
By applying Lemma~\ref{lem3.4} with $\beta=1$, we deduce that
\begin{equation}\label{cond-H1-L1-Poisson}
    \sup_{I\in\mathcal{B}_1} r_I \sup_{y\in I} \int_{(2I)^c} \|\partial_y t^{n+k}\partial_x^n\partial_t^k P^\alpha_t (x,y)\|_{L^\infty\left((0,\infty),dt\right)} d\gamma_\alpha(x)<\infty.
\end{equation}
Therefore, \ref{H1-L1-op} holds as claimed.

\subsection{Proof of Theorem~\ref{endpoints} for Littlewood-Paley functions}\label{sec-LP}

We recall that, for $n,k\in \mathbb{N}$ with $n+k\geq 1$,
\begin{equation*}
    g_{n,k}^\alpha (f)(x) = \left( \int_0^\infty \left| t^{n+k}
    \partial_x^n \partial_t^k P^\alpha_t (f) (x)\right|^2 \frac{dt}{t} 
    \right)^{1/2},
    \hspace{2pt} 
    x\in(0,\infty),
\end{equation*}
which is bounded on $L^p((0,\infty),\gamma_\alpha)$, for every $1<p<\infty$ and $n,k\in \mathbb{N}$ with $n+k\ge 1$ (\cite[Theorem 1.1, (c)]{BDQS}) and from $L^1((0,\infty),\gamma_\alpha)$ into $L^{1,\infty}((0,\infty),\gamma_\alpha)$, for every  $k\in \mathbb{N}$, $k\ge 1$, when $n=0$ (\cite[Theorem~1.2]{BdL}). 

We shall first see that, for any $n,k\in \mathbb N$, $n+k\geq 1$, $\ell=0,1$, and $x,y\in (0,\infty)$, 
\begin{align}\label{cotanorma2xy}
    &\left\|t^{n+k}\partial_y^\ell \partial_x^{n+1-\ell}\partial_t^k P^\alpha_t (x,y)\right\|_{L^2\left((0,\infty),\frac{dt}{t}\right)}\nonumber\\
    &\qquad \leq C\int_0^1 |\partial_y^\ell \partial_x^{n+1-\ell} W_{-2\log r}^\alpha(x,y)| (-\log r)^{\frac n2-1} \frac{dr}{r}.
\end{align}
By Lemma~\ref{cotasderivadasPoisson}, Minkowski's inequality, and taking $v=t^2/(-16\log r)$, we get
\begin{align*}
    &\left\|t^{n+k}\partial_y^\ell \partial_x^{n+1-\ell}\partial_t^k P^\alpha_t (x,y)\right\|_{L^2\left((0,\infty),\frac{dt}{t}\right)}\\
    &\qquad\leq C \int_0^1 \frac{\left|\partial_y^\ell \partial_x^{n+1-\ell} W_{-2\log r}^\alpha(x,y)\right|}{r(-\log r)^{\frac k2+1}}\left(\int_0^\infty t^{2(n+k)}e^{t^2/(16\log r)}\frac{dt}{t}\right)^{1/2} dr\\
    &\qquad \leq C \int_0^1 \frac{\left|\partial_y^\ell \partial_x^{n+1-\ell} W_{-2\log r}^\alpha(x,y)\right|}{r(-\log r)^{\frac k2+1}}(-\log r)^{\frac{n+k}{2}}\left(\int_0^\infty v^{n+k}e^{-v}\frac{dv}{v}\right)^{1/2} dr\\
    &\qquad =C\int_0^1 \left|\partial_y^\ell \partial_x^{n+1-\ell} W_{-2\log r}^\alpha(x,y)\right|(-\log r)^{\frac n2-1}\frac{dr}{r}
\end{align*}
whenever $n+k>0$. Thus, \eqref{cotanorma2xy} holds.

In order to prove that $g_{n,k}^\alpha$ is bounded from $L^\infty((0,\infty), \gamma_\alpha)$ to $\BMO((0,\infty), \gamma_\alpha)$, we will see that
\begin{equation}\label{cond-Linf-BMO-g}
    \sup_{I\in\mathcal{B}_1} r_I \sup_{x\in I} \int_{(2I)^c} \|\partial_x  t^{n+k}\partial_x^n\partial_t^k P^\alpha_t (x,y)\|_{L^2\left((0,\infty),\frac{dt}{t}\right)} d\gamma_\alpha(y)<\infty.
\end{equation}

From \eqref{cotanorma2xy} with $\ell=0$ we know that
\[\|\partial_x  t^{n+k}\partial_x^n\partial_t^k P^\alpha_t (x,y)\|_{L^2\left((0,\infty),\frac{dt}{t}\right)}\leq C\int_0^1 \left|\partial_x^{n+1} W_{-2\log r}^\alpha(x,y)\right|(-\log r)^{\frac n2-1}\frac{dr}{r}.\]
According to Lemma~\ref{lem3.7} and the boundedness of $\xi_n$ given in Lemma~\ref{func-acot}, we get
\begin{align*}
    &\|\partial_x  t^{n+k}\partial_x^n\partial_t^k P^\alpha_t (x,y)\|_{L^2\left((0,\infty),\frac{dt}{t}\right)}\\
&\leq C\int_0^1 \frac{r^n(-\log r)^{\frac n2-1}}{(1-r^2)^{\alpha+\frac n2+\frac32}} \int_{-1}^1 \left(1+\left|\frac{rx-ys}{\sqrt{1-r^2}}\right|^{n+1}\right) e^{-\frac{q(rx,y,s)}{1-r^2}+y^2}\Pi_\alpha(s) dsdr\\
&\leq C\int_0^1 \frac{1}{(1-r^2)^{\alpha+5/2}} \int_{-1}^1 \left(1+\left|\frac{rx-ys}{\sqrt{1-r^2}}\right|^{n+1}\right) e^{-\frac{q(rx,y,s)}{1-r^2}+y^2}\Pi_\alpha(s) dsdr.
\end{align*}
By combining Lemmas~\ref{lem3.1} and~\ref{lem3.4}, \eqref{cond-Linf-BMO-g} holds.

On the other hand, the endpoint estimate from $H^1((0,\infty), \gamma_\alpha)$ to $L^1((0,\infty), \gamma_\alpha)$ for $g_{0,k}^\alpha$, follows by taking into account the symmetry of the Poisson kernel and~\eqref{cond-Linf-BMO-g}.

\subsection{Proof of Theorem~\ref{endpoints} for Riesz transforms}\label{sec-Riesz}

According to  \cite[Theorem~1.1]{BFRS}, \cite[Theorem~1.1]{Sa4}, and \cite[Theorem~13]{No}, for every $n\in \mathbb{N}, n\ge 1$, $R_\alpha^n$ can be extended
from $L^2((0,\infty),\gamma_\alpha)\cap L^p((0,\infty),\gamma_\alpha)$ to $L^p((0,\infty),\gamma_\alpha)$ as a
bounded operator on $L^p((0,\infty),\gamma_\alpha)$, for every $1<p<\infty$.
Furthermore, $R_\alpha^1$ can be extended from $L^2((0,\infty),\gamma_\alpha)\cap L^1((0,\infty),\gamma_\alpha)$ to $L^1((0,\infty),\gamma_\alpha)$
as a bounded operator from $L^1((0,\infty),\gamma_\alpha)$ into $L^{1,\infty}((0,\infty),\gamma_\alpha)$ (see \cite{FSS2}).

Let $n\in \mathbb N$, $n\geq 1$. The kernel of the Riesz transform of order $n$ with respect to the measure $\gamma_\alpha$ is given by
\[K_\alpha^n(x,y)=\int_{-1}^1\int_0^1 r^{n-1}\left(\frac{-\log r}{1-r^2}\right)^{\frac n2-1} H_n\left(\frac{rx-ys}{\sqrt{1-r^2}}\right) \frac{e^{-\frac{q(rx,y,s)}{1-r^2}+y^2}}{(1-r^2)^{\alpha+2}}dr \Pi_\alpha(s) ds,\]
for $x,y\in (0,\infty).$ Recall that $H_n$ denotes the Hermite polynomial of order $n$ given in \eqref{Hpoly}. 

Then, from \cite[p.~62]{Leb}, we have that
\begin{align*}
    \partial_x K_\alpha^n(x,y)&=-\int_{-1}^1\int_0^1 r^{n}\left(\frac{-\log r}{1-r^2}\right)^{\frac n2-1} H_{n+1}\left(\frac{rx-ys}{\sqrt{1-r^2}}\right)\\
    &\quad \times\frac{e^{-\frac{q(rx,y,s)}{1-r^2}+y^2}}{(1-r^2)^{\alpha+5/2}}dr \Pi_\alpha(s) ds,
\end{align*}
and, therefore, by using the boundedness of $\xi_n$ given in Lemma~\ref{func-acot} we get
\[\left|\partial_x K_\alpha^n(x,y)\right|\leq C\int_{-1}^1\int_0^1 \frac{e^{-\frac{q(rx,y,s)}{2(1-r^2)}+y^2}}{(1-r^2)^{\alpha+5/2}}dr \Pi_\alpha(s) ds.\]
According to Lemma~\ref{lem3.1} and \ref{eq-D0},
\begin{equation*}
    \sup_{I\in\mathcal{B}_1} r_I \sup_{x\in I} \int_{(2I)^c} \left|\partial_x K_\alpha^n(x,y)\right| d\gamma_\alpha(y)<\infty
\end{equation*}
so $R_\alpha^n$ is bounded from $L^\infty((0,\infty),\gamma_\alpha)$ to $\BMO((0,\infty),\gamma_\alpha)$ by virtue of Theorem~\ref{criterion}, for any $n\in \mathbb N$ with $n\geq 1$.

On the other hand, for $n=1$, by \ref{eq-D0} we obtain
\begin{align*}
    \partial_y K_\alpha^1(x,y)&=e^{x^2}\int_{-1}^1\int_0^1 \frac{1}{(1-r^2)^{\alpha+2}}\left(\frac{-\log r}{1-r^2}\right)^{-\frac12}\\
    &\quad \times \partial_y\left[2\left(\frac{rx-ys}{\sqrt{1-r^2}}\right)e^{-\frac{q(x,ry,s)}{1-r^2}}\right] dr \Pi_\alpha(s) ds\\
    &=-2e^{x^2}\int_{-1}^1\int_0^1 \frac{\varphi(r)^{1/2}}{(1-r^2)^{\alpha+2}}e^{-\frac{q(x,ry,s)}{1-r^2}}\\
    &\quad \times \left(\frac{s}{\sqrt{1-r^2}}+\frac{2r(rx-ys)(ry-xs)}{(1-r^2)^{3/2}}\right) dr \Pi_\alpha(s) ds.
\end{align*}
Hence, by Lemma~\ref{func-acot} we can estimate
\begin{align*}
    \left|\partial_y K_\alpha^1(x,y)\right|&\leq C\int_{-1}^1\int_0^1 \frac{e^{-\frac{q(x,ry,s)}{1-r^2}+x^2}}{(1-r^2)^{\alpha+5/2}}dr \Pi_\alpha(s) ds\\
&\quad +C\int_{-1}^1\int_0^1 \frac{r|rx-ys||ry-xs|}{(1-r^2)^{\alpha+7/2}}e^{-\frac{q(x,ry,s)}{1-r^2}+x^2} dr \Pi_\alpha(s) ds.
\end{align*}
By combining Lemmas~\ref{lem3.1} and \ref{lem3.3}, we get
\begin{equation*}
    \sup_{I\in\mathcal{B}_1} r_I \sup_{y\in I} \int_{(2I)^c} \left|\partial_y K_\alpha^n(x,y)\right| d\gamma_\alpha(x)<\infty
\end{equation*}
so \ref{H1-L1-op} holds for $R_\alpha^1$.

\subsection{Proof of Theorem~\ref{endpoints} for multipliers of Laplace transform type}\label{sec-mult}
Suppose that $M$ is of Laplace transform type given by
\begin{equation*}
    M(x) = x \int_{0}^\infty e^{-xy} \phi(y) dy, \quad x\in(0,\infty),
\end{equation*}
where $\phi\in L^{\infty}(0,\infty)$. We consider the function
\begin{equation*}
    K_\phi^\alpha(x,y) = \int_0^\infty \phi(t) \left(
    -{\partial_t}\right) P^\alpha_t(x,y)dt,
    \quad x,y\in(0,\infty).
\end{equation*}
Note here that $K^\alpha_\phi(x,y) = K^\alpha_\phi(y,x)$ for $x,y\in(0,\infty)$. 

We have that, for $x,y\in(0,\infty)$, $x\neq y$,
\begin{equation*}
    \partial_x K^\alpha_\phi (x,y) = \int_0^\infty \frac{\phi(t)}{\sqrt{2\pi}} \int_0^1 \partial_t \left(te^{\frac{t^2}{8\log r}}\right)
    \partial_x W^\alpha_{-2\log r}(x,y) \frac{dr}{r(-\log r)^{3/2}}
    dt.
\end{equation*}
By using Lemma~\ref{cotasderivadasPoisson} with $\ell=n=0$ and $k=1$, and computing the inner integral, we get
\begin{equation*}
    \begin{split}
        |\partial_x K^\alpha_\phi(x,y)| & \leq 
        C  \int_0^1 \int_0^\infty e^{\frac{t^2}{16\log r}} dt \frac{1}{r (-\log r)^{3/2}}
    |\partial_x W^\alpha_{-2\log r}(x,y) | dr
        \\ & \leq C
        \int_0^1 |\partial_x W^\alpha_{-2\log r}(x,y) | 
        \frac{dr}{r (-\log r)},
    \end{split}
\end{equation*}
for every $x,y\in(0,\infty)$, $x\neq y.$

Continuing as in Section~\ref{sec-maximal} we can find $C>0$ such that for every  $I\in\mathcal{B}_1$,
\begin{equation*}
    \sup_{x\in I} r_I \int_{(2I)^c} |\partial_x K^\alpha_\phi(x,y)| d\gamma_\alpha(y) \leq C.
\end{equation*}

Since $T_M$ is bounded from $L^2((0,\infty),\gamma_\alpha)$ into itself, we can obtain the conclusion of Theorem~\ref{endpoints}, (a), for the multiplier $T_M$.

The operator $T_M$ is self-adjoint in $L^2((0,\infty),\gamma_\alpha)$. Then, a duality argument (see \cite[Lemma 7.1]{MMS-end}) allows us to conclude for $T_M$ the property in Theorem~\ref{endpoints}~\ref{H1-L1-op}.

\subsection{Proof of Theorem~\ref{endpoints} for variation operators}\label{sec-var}

The $\rho$-variation operator $V_\rho(\{t^k\partial_t^k P_t^\alpha\}_{t>0})$, $k\in \mathbb{N}$, is bounded on $L^p((0,\infty),\gamma_\alpha)$, ${1<p<\infty}$ (\cite[Corollary~6.1]{LeMX1} and \cite[Theorem~3.3]{JR}) and from $L^1((0,\infty),\gamma_\alpha)$ into $L^{1,\infty}((0,\infty),\gamma_\alpha)$ (\cite[Theorem~1.1]{BdL}).

Assume $g\in C^1(0,\infty)$. If $0<t_N<t_{N-1}<\dots<t_1$ we have that
\begin{equation*}
    \begin{split}
        \left(
        \sum_{j=1}^{N-1} |g(t_j)-g(t_{j+1})|^\rho
        \right)^{1/\rho} & 
        \leq 
        \sum_{j=1}^{N-1} |g(t_j)-g(t_{j+1})|
        \leq \sum_{j=1}^{N-1} \left|\int_{t_{j+1}}^{t_j} g'(t)dt\right|
        \\ & \leq \int_0^\infty |g'(t)|dt.
    \end{split}
\end{equation*}
Then, 
\begin{equation*}
    V_\rho\left(\{g(t)\}_{t>0}\right) \leq 
    \int_{0}^\infty |g'(t)|dt.
\end{equation*}

Hence, for every $k\in\mathbb{N}$ we get
\begin{equation*}
\begin{split}
    V_\rho \left(\{ \partial_x t^k \partial_t^k P^\alpha_t(x,y)\}_{t>0}\right) 
    & \leq 
    \int_0^\infty | \partial_t (\partial_x t^k \partial_t^k P^\alpha_t(x,y))| dt
    \\ & \leq k \int_0^\infty
     | t^{k-1} \partial_x \partial_t^k P^\alpha_t(x,y)|dt\\
     &\quad +\int_0^\infty
     | t^{k} \partial_x \partial_t^{k+1} P^\alpha_t(x,y)|dt.
\end{split} 
\end{equation*}

According to Lemma~\ref{cotasderivadasPoisson} and proceeding as in the previous section, we obtain 
\begin{equation*}
    \begin{split}
        \int_0^\infty
     | t^{k} \partial_x \partial_t^{k+1} P^\alpha_t(x,y)| dt
     & \leq 
        C \int_0^\infty t^k \int_0^1  \frac{e^{\frac{t^2}{16\log r}}}{r (-\log r)^{(k+3)/2}}
    |\partial_x W^\alpha_{-2\log r}(x,y) | drdt
        \\ & \leq C
        \int_0^1 |\partial_x W^\alpha_{-2\log r}(x,y) | 
        \frac{dr}{r (-\log r)},
    \end{split}
\end{equation*}
for every $x,y\in(0,\infty)$, $x\neq y.$

The other term can be controlled by proceeding in a similar way.

Therefore, as in the last section, we obtain the conclusion of Theorem~\ref{endpoints}  for the $\rho$-variation operator $V_\rho \left(\{ \partial_x t^k \partial_t^k P^\alpha_t(x,y)\}_{t>0}\right)$.

\subsection{Proof of Theorem~\ref{endpoints} for fractional integrals}\label{sec-fracint}
Let $\omega>0$. Recall that the  integral kernel of the operator $\Delta_\alpha^{-\omega}$ is given by
\begin{align*}
    H^\omega_\alpha(x,y)& = 
    \frac{1}{\Gamma(\omega)} \int_0^1 
    \left( \int_{-1}^1 e^{-\frac{q(rx,y,s)}{1-r^2}+y^2} \Pi_\alpha(s) ds \frac{1}{(1-r^2)^{\alpha+1}} - 1
    \right)\\
    &\quad \times\frac{dr}{r (-\log r)^{1-\omega}},  
\end{align*}
for $x,y\in(0,\infty)$, $x\neq y$. Notice that $ H^\omega_\alpha(x,y) =  H^\omega_\alpha(y,x)$.

We can write, for $x,y\in(0,\infty)$, $x\neq y$,
\begin{equation*}
    \partial_x H^\omega_\alpha(x,y)  = 
    \frac{2}{\Gamma(\omega)} \int_0^1 \int_{-1}^1 e^{-\frac{q(rx,y,s)}{1-r^2}+y^2} 
    \frac{rx-ys}{(1-r^2)^{\alpha+2}}
    \frac{\Pi_\alpha(s) dsdr}{(-\log r)^{1-\omega}}
\end{equation*}
Then, according to~\ref{eq-D6} we get
\begin{equation*}
    |\partial_x H^\omega_\alpha(x,y)| \leq  C\int_0^1 \int_{-1}^1 e^{-\frac{q(rx,y,s)}{1-r^2}+y^2} \Pi_\alpha(s) ds
    \frac{(-\log r)^{\omega-1}}{(1-r^2)^{\alpha+3/2}}dr  
\end{equation*}

By using Lemma~\ref{lem3.1.1} and~\ref{eq-D0} we can find $C>0$ such that
\begin{equation*}
    \sup_{x\in I}r_I \int_{(2I)^c} |\partial_x H^\omega_\alpha(x,y)| d\gamma_\alpha(y) \leq C
\end{equation*}
for every $I\in \mathcal{B}_1$. Since $\Delta^{-\omega}$ is bounded in $L^2((0,\infty),\gamma_\alpha)$, the property that we have just proved, together with the symmetry of the kernel, allow us to obtain the desired conclusion using Theorem~\ref{criterion} and duality.

\subsection*{Statements and Declarations} 

\subsubsection*{Funding} The first author is partially supported by grant PID2019-106093GB-I00 from the Spanish Government. The second and fourth authors are partially supported by grants PICT-2019-2019-00389 (Agencia Nacional de Promoción Científica y Tecnológica) and CAI+D 2019-015 (Universidad Nacional del Litoral)

\subsubsection*{Competing Interests} The authors have no relevant financial or non-financial interests to disclose.

\subsubsection*{Author Contributions} All authors whose names appear on the submission made substantial contributions to the conception and design of the work, drafted the work and revised it critically, approved this version to be published, and agree to be accountable for all aspects of the work in ensuring that questions related to the accuracy or integrity of any part of the work are appropriately investigated and resolved.

\subsection*{Data availability} Data sharing not applicable to this article as no datasets were generated or analysed during the current study.

\bibliographystyle{acm}

\end{document}